\documentclass[11pt]{article}
\usepackage{Preamble}
\usepackage{xcolor}

\begin{document}

\author{ 
Garrett Nelson}

\title{Rational parking functions and $(m, n)$-invariant sets}

\maketitle

\abstract{

    An $(m, n)$-parking function can be characterized as function $f:[n] \to [m]$ such that the partition obtained by reordering the values of $f$ fits inside a right triangle with legs of length $m$ and $n$. Recent work by McCammond, Thomas, and Williams define an action of words in $[m]^n$ on $\nR^n$. They show that rational parking functions are exactly the words that admit fixed points under that action. An $(m, n)$-invariant set is a set $\Delta \subset \nZ$ such that $\Delta + m \subset \Delta$ and $\Delta + n \subset \Delta$. In this work we define an action of words in $[m]^n $ on $(m, n)$-invariant sets by removing the $j$th $m$-generator from $\Delta$. We show this action also characterizes $(m, n)$-parking functions. Further we show that each $(m, n)$-invariant set is fixed by a unique monotone parking function. By relating the actions on $\nR^m$ and on $(m, n)$-invariant sets we prove that the set of all the points in $\nR^m$ that can be fixed by a parking function is a union of points fixed by monotone parking functions. In the case when $\gcd(m, n) =1$ we characterize the set of periodic points of the action defined on $\nR^m$ and show that the algorithm reversing the Pak-Stanley map proposed by Gorsky, Mazin, and Vazirani converges in a finite amount of steps.

}

\section{Introduction }

Parking functions were first defined in \cite{Konheim_G_Weiss_Occupancy_Discipline}, and the generalization to rational parking functions was studied in \cite{gorsky_mazin_vazirani_2014,ARMSTRONG2015159}. In \cite{mccammond_thomas_williams_2019} Jon McCammond, Hugh Tomas, and Nathan Williams give a new characterization of $(m, n)$-parking functions. They define an action of words in $[m] = \{ 0, 1, \dots, m-1\}$ on the Weyl Chamber 
\[
V^m = \{ \x = (x_0, x_1, \dots , x_{m-1}) : \sum_i x_i = 0 \textnormal{ and } x_0 \leq x_1 \leq \dots \leq x_{m-1}\},
\]
where each letter $i$ acts on $\x \in V^m$ by adding $m$ to $x_i$, subtracting $1$ from each coordinate, then resorting the coordinates. Then a word $w \in [m]^n$ acts on $\x \in V^m$ from right to left. The authors of \cite{gorsky_mazin_vazirani_2017} prove that $\w$ is an $(m, n)$-parking function if and only if $\w$ has a fixed point in $V^m$. They further show that the set of fixed points is a convex set of affine dimension $\gcd(m, n) -1$.

In \cite{gorsky_mazin_vazirani_2017} Eugene Gorsky, Mikhail Mazin, and Monica Vazirani study $(m, n)$-Dyck paths when $m$ and $n$ are not relatively prime and $(m, n)$-invariant sets of $\nZ$. A set $\Delta \subset \nZ$ is $(m, n)$ invariant if $\Delta + m \subset \Delta$ and $\Delta + n \subset \Delta$. When $m$ and $n$ are coprime, there is a bijection from $(m, n)$-Dyck paths to 0-normalized $(m, n)$-invariant sets. When $m$ and  $n$ are not coprime the authors define an equivalence relation on $(m, n)$-invariant sets such that the equivalence classes are in bijection with $(m, n)$-Dyck paths.

In this paper, we connect the ideas from \cite{mccammond_thomas_williams_2019} and \cite{gorsky_mazin_vazirani_2017}. We define an action of words in $[m]^n$ on  bounded and co-bounded $ m $-invariant sets. An $(m, n)$-invariant set $\Delta$ is bounded and co-bounded if there are an $N, M \in \nN$ such that  $\nZ_{\leq -N} \cap \Delta = \emptyset$ and $\nZ_{\geq M} \subset \Delta $. The $m$-generators of $\Delta$ are given be $(\Delta + m) \setminus \Delta = \{ a_1 < \dots < a_m \}$. The letter $i \in [m]$ will act on $\Delta$ by removing the $a_i$ from $\Delta$. A word in $[m]^n$ acts from right to left. We prove:

\begin{theorem}
    Let $\w$ be a word in $[m]^n$ and $\Delta$ an $m-$invariant set such that $\w \cdot \Delta = \Delta + n$ then $\w$ is an $(m, n)$-parking function and $\Delta$ is $(m, n)$ invariant. Furthermore for any $(m, n)-$parking function there exists an $(m, n)$-invariant set such that $\w \cdot \Delta = \Delta + n$.  
\end{theorem}

We connect this action to the bijection $\cG$ defined in \cite{gorsky_mazin_vazirani_2017} from the equivalence classes of $(m, n)$-invariant sets to the set of $(m, n)$-Dyck paths.

\begin{theorem}

    The monotone parking function given by the column heights of $\cG(\Delta)$ is the unique monotone parking function such that $\w  \cdot \Delta = \Delta + n$.
\end{theorem}
When $\gcd(m, n) =1$ we characterize the points in $V^m$ that are periodic under the action of a parking function. 

\begin{theorem}
    If $\gcd(m, n) = 1$ and $\p$ an $(m, n)$-parking function with fixed point $\x_p \in V^m$, then $\x_p$ is the centroid of an alcove $A_p$ of the hyperplane arrangement $\cH = \{ x_i - x_j = km , 1 \leq i < j \leq n, k \in \nZ\}$. The set periodic points of $\p$ is exactly $A_p\setminus x_p$ and all points in $A_p\setminus x_p$ have period $n$.
\end{theorem}

Structure of this paper: In Section 2, we provide background on parking functions, the action of words in $[m]^n$ on $V^m$ from \cite{mccammond_thomas_williams_2019}, the affine braid arrangement, and invariant sets. In Section 3, we define an action on $m$-invariant sets and then prove Theorem 1.1. In Section 4, we focus on monotone parking functions and prove Theorem 1.2. In Section 5, we show that the set of fixed points of a parking function is the union of fixed points of monotone parking functions. In Section 6, we consider the case when $\gcd(m, n) =1$ and classify all the periodic points of an $(m, n)$-parking function proving Theorem 1.3. In section 7, as a corollary of Theorem 1.3 we are able to resolve conjecture $7.9$ from \cite{gorsky_mazin_vazirani_2014} showing an algorithm to invert the map $SP$ does finish in a finite number of steps. 

\section{Background}

\subsection{Rational Parking Functions:}

Classical parking functions arise as a list of preferences from parking cars, first studied by Konheim and Weiss in \cite{Konheim_G_Weiss_Occupancy_Discipline}. Imagine $n$ cars driving on a one-way street with $n$ parking spots ahead. Each car has a preference $w_i = f(i) \in [n] = \{ 0, 1, \dots n-1\}$ for what parking spot it wants to park in. Each car tries to park in its preferred spot. If the preferred parking spot is available, it does; if not, the car parks at the first available parking spot after. If all cars can park, then the function $f: [n] \to [n]$ is a parking function. We write these functions as words $\w = w_{n-1}w_{n-2} \dots w_1w_0$ where $w_i = f(i)$. The requirement that all cars park can be summarized by 

\[
\# \{ j : w_j < i \} \geq i \textnormal{ for } 1 \leq i \leq n.
\]
Any permutation of a parking function is also a parking function. 

\begin{definition}
    A Dyck word is a parking function where $w_0 \leq w_1 \leq \dots \leq w_{n-1}$.
\end{definition}

    Dyck words are named after Dyck paths.
    
\begin{definition}
    A Dyck path is a lattice path on a $n\times n$ grid from $(0,n)$ to $(n,0)$ where the path stays weakly below the main diagonal.
\end{definition}    

    The correspondence between Dyck paths and Dyck words is given by the column height of each horizontal step read left to right. Classical parking functions are permutation of the column heights of Dyck paths.

\vspace{10 pt}

\begin{figure}[ht]
\centering

\begin{subfigure}{0.4\textwidth}

      \begin{center}

\begin{tikzpicture}[scale=0.4]
\draw  [very thin, gray](0,0) grid (9,6);
\draw [very thin, gray] (0, 6) --(9,0);

\draw  (8.5,-.5) node {$0$}; 
\draw  (7.5,-.5) node {$0$}; 
\draw  (6.5,-.5) node {$0$}; 
\draw  (5.5,-.5) node {$2$}; 
\draw  (4.5,-.5) node {$2$}; 
\draw  (3.5,-.5) node {$2$}; 
\draw  (2.5,-.5) node {$2$}; 
\draw  (1.5,-.5) node {$2$}; 
\draw  (0.5,-.5) node {$3$};

\draw [very thick, blue] (9, 0)--(6, 0) --(6, 2) --(1, 2) --(1, 3) --(0, 3) --(0, 6);

\end{tikzpicture}

\end{center}

\caption{ Dyck Path $322222000$ in an $6 \times 9$ grid. }  \label{Figure:RationalDyck}

\end{subfigure}
\hspace{ 10 pt}
\begin{subfigure}{0.4\textwidth}
\begin{center}
\begin{tabular}{ccccc}
000 & 001 & 002 & 011 & 012\\
    & 010 & 020 & 101 & 021\\
    & 100 & 200 & 110 & 102\\
    &     &     &     & 120\\
    &     &     &     & 201\\
    &     &     &     & 210
\end{tabular}
\end{center}
\caption{All $(4, 3)$-parking functions.   }  \label{Figure:All43Parking}

\end{subfigure}

\end{figure}

    We construct rational parking functions by changing the square grid of an Dyck path to an integral rectangle. 

\begin{definition}
    Lattice paths in an $m \times n$ grid from $(0, m)$ to $(n,0)$ that stay weakly below the main diagonal are called $(m, n)$-Dyck paths (see \cref{Figure:RationalDyck}).    
\end{definition}
 
 The column heights of such paths are $(m, n)$-Dyck Words. Permutations of $(m, n)$-Dyck words are $(m, n)$-parking functions. 

\begin{definition}
    An $(m, n)$-parking functions is a word $\w = w_{n-1}w_{n-2}\dots w_1w_0  \in [m]^n$ such that
    \[
        \# \{ j : w_j < i \} \geq \frac{i n}{m} \textnormal{ for } 1 \leq i \leq m.
    \]
\end{definition}

We denote the set of all $(m, n)$-parking functions as $PW_n^m$. In \cref{Figure:All43Parking} all the $(4, 3)-$parking functions in $PW_4^3$ are listed. These parking functions were first studied in \cite{ARMSTRONG2015159,gorsky_mazin_vazirani_2014}.

Often $m$ and $n$ will be coprime. Classical parking functions can be recovered from the coprime case by considering $(n+1, n)$-parking functions. In parts of this paper we consider when $m$ and $n$ are not coprime. Aval and Bergeron have started calling these parking functions rectangular parking functions in \cite{Aval_Bergeron_2015}. We will make it clear when a result holds for any parking functions or just when $m$ and $n$ are coprime.  

\subsection{Acting by a Parking function:}
In \cite{mccammond_thomas_williams_2019} Jon McCammond, Hugh Tomas, and Nathan Williams build a new characterization of $(m, n)$-parking functions based on an action. Words in $\{ 0, 1, \dots , m-1\} = [m]$ act on the Weyl chamber 
\[
V^m = \{ \x = (x_0, x_1, \dots , x_{m-1}) \in \nR^m : \sum_i x_i = 0 \textnormal{ and } x_0 \leq x_1 \leq \dots \leq x_{m-1}\},
\]
where $i$ acts on $\x$ by adding $m$ to the $i$th component, subtracting 1 from every component, then resorting. Differing from \cite{mccammond_thomas_williams_2019} we use the convention that words act one letter at a time from right to left. If $w \in [m]^n$ we can think of $\w$ as a piecewise-linear function from $V^m$ to $V^m$.

\begin{theorem}[\cite{mccammond_thomas_williams_2019}\label{Theorem: Fixed iff Parking}, Theorem 1.1]
Let $w \in [m]^n$, then $\w$ is an $(m, n)$-parking functions if and only if $\w$ has a fix point in $V^m$. The set of fixed points is convex of affine dimension $\gcd(m, n) - 1$  and contained in an affine subset of dimension $\gcd(m, n) -1$. In particular, when $m$ and $n$ are coprime, each parking function has a unique fixed point. 
\end{theorem} 

\subsection{The Affine Braid Arrangement} \label{Affine Braid}


There are various bijections from parking functions to regions or alcoves of hyperplane arrangements. 
Let $H_{i, j}^k=\{ x_i - x_j = kn \} $.

\begin{definition}
    The affine braid arrangement is the hyperplane arrangement $\cH = \{ H_{i,j}^k : 0 \leq i < j \leq n-1 , k \in \nZ \}$ inside of $W =  \{ \x \in \nR^n : \sum_{i=0}^{n-1} x_i = 0 \} $.
\end{definition}
    Note that we scale our hyperplane arraignments by $n$. The regions of the hyperplane arrangement $H$ are the connected components of $W \setminus H$. The closure of a region of the affine braid arrangement is referred to as an alcove. 

\begin{definition}
    The fundamental alcove $A_0$ is the alcove with boundary conditions given by $x_0 \leq x_1 \leq \dots \leq x_{n-1} \leq x_0 + n$ inside $W$. 
\end{definition}
    Note the fundamental alcove has boundaries given by $H_{i-1, i}^0$ for $i \in \{ 1, \dots , n-1\}$ and $H_{n-1, 0}^1$. 

\begin{definition}
    The height of $H_{i, j}^k$ is $|j - i - nk|$. Note that the height $H_{i, j}^k$ is proportional to the distance between $H_{i, j}^k$ and the center of $A_0$. 
\end{definition}    
      The hyperplanes that bound the fundamental alcove are exactly those of height 1.

\begin{definition}[\cite{Shi_1987}]
    The Shi arrangement is $Shi_n = \{ H_{i, j}^k:  0 \leq i < j \leq n-1, k =0 \text{ or } k = 1\} $.
\end{definition}
    The Shi arrangement is a sub hyperplane arrangement of the affine braid arrangement. The regions of $Shi_n$ were shown to be in bijection with classical parking functions in \cite{Stanley_1998} by Stanley with correspondence from Pak.

\begin{definition}
    The $k$-$Shi$ arrangement is given by $k$-$Shi_n = \{ H_{i, j}^k:  0 \leq i < j \leq n-1, k \in \{ -k+1, -k+2, \dots,k\}  \}$
\end{definition}

    The Pak-Stanley construction extends to the $k$-$Shi$ case, however, the argument is rather technical. For a more geometric approach, see \cite{Mazin_2017}, which took inspiration from the work of Hopkins and Perkinson in \cite{Hopkins_Perkinson_2015}. For a more general bijection, we need to introduce the Sommers region. 

\begin{definition}[\cite{Sommers_2005}]
    The Sommers region $\cS_n^m$ is the region in $W$ bounded by all hyperplanes $H_{i, k}^k$ of height $m$.
\end{definition}

\begin{theorem}[\cite{mccammond_thomas_williams_2019}, Theorem 6.20]
    The alcoves inside the Sommers region are in bijection with parking functions in $PW_n^m$. 
\end{theorem}

    The Sommers region and $k$-$Shi$ arrangements are related through the affine symmetric group. 

\begin{definition}
    The affine symmetric group $\widetilde{S_n}$ consists of bijections $\omega: \nZ \to \nZ$ such that 
\begin{align*}
    \text{1) }&\omega(x+n) = \omega(x) + n & \text{2)}&\sum_{k=0}^{n-1} \omega(k) = \frac{(n-1)n}{2}.
\end{align*}
\end{definition}

We write elements of $\widetilde{S_n} $ in window notation $\omega = [\omega(0), \omega(1), \dots , \omega(n-1)]$. The generators $\widetilde{S_n}$ are given by $s_i = [0, 1, \dots ,i-1, i+1, i, i+2,\dots, n-1]$ for $i \in \{1, \dots , n-1\}$ and $s_0 = [-1, 1,, 2 \dots, n-2, n]$. 
$\widetilde{S_n}$ acts on $W$ with the generator $s_i$ acting by reflection across the hyperplane $H^0_{i-1, i}$ for $0 < i < n$ and $s_0$ acting by reflecting across the hyperplane $H^1_{n-1, 0}$. This action sends alcoves to alcoves is free, and is transitive on alcoves. The map $\omega \to \omega(A_0)$ gives a bijection from $\widetilde{S_n}$ to the alcoves of the affine braid arrangement.

\begin{figure}[ht]
\begin{center}

    \begin{tikzpicture}[scale = 1.6]
\draw    (-1.5,-0.866025) -- (2.5 ,-0.866025) -- (.5,3*0.866025) --(-1.5,-0.866025)  ;
\draw[thick] (-1, 0) -- (2, 0);
\draw (-.5, .866025) -- (1.5, .866025);
\draw (0, .866025*2) -- (1, .866025*2);
\draw[thick] (0, .866025*2) -- (1.5, .866025*-1);
\draw (-.5, .866025) -- (.5, .866025*-1);
\draw (-1, 0) -- (-.5, .866025*-1);
\draw[thick] (-.5, -.866025) -- (1., .866025*2);
\draw (.5, -.866025) -- (1.5, .866025);
\draw (1.5, -.866025) -- (2, 0);

\node at (.5, .866025*2.4)    {\tiny $[4$ $-13]$};
\node at (0, .866025*1.4)    {\tiny $[-134]$};
\node at (.5, .866025*1.6)    {\tiny $[-143]$};
\node at (1, .866025*1.4)    {\tiny $[042]$};
\node at (-.5, .866025*.4)    {\tiny $[312]$};
\node at (0, .866025*.6)    {\tiny $[132]$};
\node at (.5, .866025*.4)    {\tiny $[123]$};
\node at (1, .866025*.6)    {\tiny $[024]$};
\node at (1.5, .866025*.4)    {\tiny $[204]$};
\node at (-1, .866025*-.6)    {\tiny $[-226]$};
\node at (-.5, .866025*-.4)    {\tiny $[321]$};
\node at (0, .866025*-.6)    {\tiny $[231]$};
\node at (.5, .866025*-.4)    {\tiny $[213]$};
\node at (1, .866025*-.6)    {\tiny $[015]$};
\node at (1.5, .866025*-.4)    {\tiny $[105]$};
\node at (2, .866025*-.6)    {\tiny $[150]$};

\end{tikzpicture}

\caption{Here is $\cS_3^4$ with alcoves labeled by affine permutations. The fundamental alcove is labeled by $[123]$. } 
\label{Figure:Sommers34}

\end{center}
\end{figure}
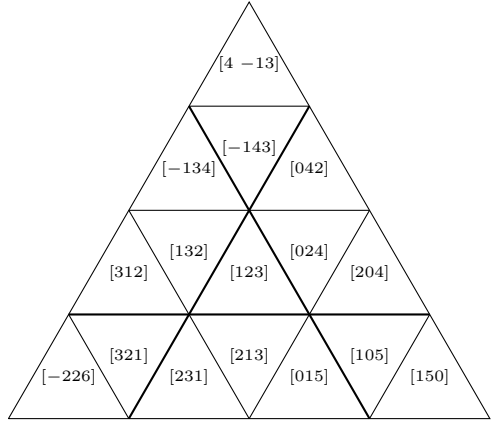

 The affine permutations that correspond to alcoves inside the Sommers region are call $m$-restricted, the set of these affine permutations is denoted ${}^m\widetilde{S_n}$, it can be seen that $${}^m\widetilde{S_n} = \{ \omega \in \widetilde{S_n} : \omega^{-1}(x) < \omega^{-1}(x+m), \forall x \in \nZ\}.$$ If 
 $\omega^{-1} \in { }^m \widetilde{S_n} $ then it is called $m$-stable. The set of $m$-stable affine permutations is $$\widetilde{S_n^m} = \{  \omega \in \widetilde{S_n} : \omega(x) < \omega(x+m), \forall x \in \nZ \}.$$

\begin{theorem}[\cite{Fishel_Vazirani_2010}, Theorem 7.1]
    The alcoves inside $\cS_n^{kn+1}$ region are in bijection with the regions of the $k$-$Shi_n$. The bijection takes an alcove labeled by $\omega \in \widetilde{S_n^m}$ and maps it to the region containing the alcove labeled by $\omega^{-1}$. The alcove labeled by $\omega^{-1}$ is the unique alcove in that region with the least number of hyperplane $H_{i, j}^k$ between it and $A_0$. 
\end{theorem}

\begin{definition} \label{def: Centriod}
    The centroid of an alcove is the unique point $(x_0, x_1, \dots x_{m-1} )$ in the interior of the alcove such that all the numbers
    $$
    \frac{m+1}{2} - mx_0, \quad \frac{m+1}{2} - mx_1, \quad \frac{m+1}{2} - mx_{m-1}
    $$
    are all integers. 
\end{definition}

\begin{theorem}[Lemma 2.9 \cite{gorsky_mazin_vazirani_2014}]
    Centroids of alcoves exist. Moreover, if $(x_0, \dots, x_{m-1}) \in \omega(A_0)$ is the centroid of $\omega(A_0)$ then 
    \[
    \omega^{-1} = [ \frac{m+1}{2} - mx_0, \quad \frac{m+1}{2} - mx_1, \quad \frac{m+1}{2} - mx_{m-1} ]
    \]

\end{theorem}

\subsection{Alcoves and the Action of a Parking Function:} \label{SubSec: Alcoves and the action of a Parking Function}

We return to a few results about how the action of parking functions relates to alcoves. 

\begin{theorem}[\cite{mccammond_thomas_williams_2019}, Lemma 2.3]
    The action of a word $\w \in [m]^n$ on $V^m$ is a weak contraction. If $A$ is an alcove, then $w|_A : A \to w(A)$ is a linear isometry and $w(A)$ is an alcove.
\end{theorem}

 Indeed, if $\x, \y \in A$ then acting by letter $k$ translates $\x$ and $\y$ to the same alcove, where they are on the same side of every simple hyperplane $H^0_{i-1, i}$. Thus, these points will be resorted in the same way.

\begin{theorem}[\cite{mccammond_thomas_williams_2019}, Lemma 2.3]
    When $m$ and $n$ are coprime, the unique fixed point of a parking function $\p \in PW_n^m$ is the centriod of an alcove. 
\end{theorem}

In \cite{gorsky_mazin_vazirani_2014} the authors define a Pak-Stanely map $SP : {}^m\widetilde{S_n} \to PF_n^m$ given by

\[
\mathcal{S P}_\omega(\alpha):=\sharp\left\{\beta \mid \beta>\alpha, 0<\omega(\beta)-\omega(\alpha)<m\right\}.
\]

In conjecture $7.9$ of \cite{gorsky_mazin_vazirani_2014} the authors give a possible algorithm to invert $SP$. In \cite{mccammond_thomas_williams_2019} the authors show $SP$ is a bijection. \cref{corollary: x0 converge to fixed point } shows that the algorithm from \cite{gorsky_mazin_vazirani_2014} converges.

\subsection{Invariant sets}

We follow the definitions and notation from \cite{gorsky_mazin_2013}.

\begin{definition}
    A set $\Delta \subset \nZ$ is an $n$-invariant set if $\Delta + n \subset \Delta$. A set $\Delta \subset \nZ$ is an $(m, n)$-invariant set $\Delta$ if it is both an $m$-invariant set and $n$-invariant set. Denote the set of all $(m, n)$-invariant sets with $M_{m, n}$.
\end{definition}   

\begin{definition}
    An invariant set $\Delta$ is called bounded if there exists a minimum element of $\Delta$. An invariant set $\Delta$ is called co-bounded if there is some $k \in \nZ$ such that $\nZ_{\geq k} \subset \Delta$.
\end{definition}

    See that being co-bounded is the same as $-(\nZ \setminus \Delta)$ begin bounded. 
    
\begin{definition}
    If $\Delta$ is an $n$-invariant set we say $a \in \Delta$ is an \textit{$n$-generator} if $a - n \notin \Delta$. We say $a \notin \Delta$ is an \textit{$n$-cogenerator} if $a + n \in \Delta$. 
     
\end{definition}
See that the $n$-cogenerators of $\Delta$ are $(-n)$-generators of $\nZ\setminus \Delta$. For any $k \in \nZ$ if $\Delta$ is an $(m, n)$-invariant set, then $\Delta + k$ is still an $(m, n)$-invariant set. We often consider $(m, n)$-invariant sets up to these shifts. We have several choices for a representative of each class.

\begin{definition}
    An $(m, n)$-invariant set $\Delta$ is $0$-normalized if $\min(\Delta) = 0$. We say $\Delta$ is $0$-balanced if $\#  \nZ_{\geq 0} \setminus \Delta  = \# \nZ_{<0} \cap \Delta $. 
\end{definition}

\begin{definition}
    The skeleton of $\Delta \in M_{m, n}$ is the union of $m$-generators and $n$-cogenerators. 
\end{definition}

When $\gcd(m, n) = 1$ we can visualize $(m, n)$-invariant sets as periodic lattice paths. Consider a lattice with boxes labeled by $l(x, y) = -nx - my $, where $(x, y)$ are coordinates of the northeast corner of each box. See that the labeling is periodic with period $(m, -n)$. Given an $(m,n)$-invariant set $\Delta$ the corresponding periodic path $P(\Delta)$ is the boundary of the union of boxes labeled by integers in $\Delta$. The $m$-generators are the labels just below $P(\Delta)$ and the $n$-cogenerators are the labels just right of $P(\Delta)$. The skeleton of $\Delta$ are the labels of all boxes with the northeast corner on $P(\Delta). $ (see \cref{Figure:periodic_path}). 

\begin{figure}[ht]
    \begin{center}

    \begin{tikzpicture}[scale=0.75]
\draw  [very thin, gray](0,1) grid (11,10);
\draw [red] (0.5,9.5) node {$0$}; 
\draw (0.5,8.5) node {$5$}; 
\draw (0.5,7.5) node {$10$}; 
\draw (0.5,6.5) node {$15$};

\draw [blue](1.5,9.5) node {$-4$}; 
\draw [blue](1.5,8.5) node {$1$}; 
\draw [red] (1.5,7.5) node {$6$}; 
\draw (1.5,6.5) node {$11$}; 
\draw (1.5,5.5) node {$16$}; 

\draw (2.5,9.5) node {$-8$}; 
\draw (2.5,8.5) node {$-3$}; 
\draw [blue](2.5,7.5) node {$2$}; 
\draw [red] (2.5,6.5) node {$7$}; 
\draw (2.5,5.5) node {$12$};

\draw (3.5,8.5) node {$-7$}; 
\draw (3.5,7.5) node {$-2$}; 
\draw[red] (3.5,6.5) node {$3$}; 
\draw (3.5,5.5) node {$8$}; 
\draw (3.5,4.5) node {$13$};

\draw (4.5,7.5) node {$-6$}; 
\draw [blue](4.5,6.5) node {$-1$}; 
\draw [red] (4.5,5.5) node {$4$}; 
\draw (4.5,4.5) node {$9$}; 
\draw (4.5,3.5) node {$14$};

\draw (5.5,6.5) node {$-5$}; 
\draw [red] (5.5,5.5) node {$0$}; 
\draw (5.5,4.5) node {$5$}; 
\draw (5.5,3.5) node {$10$}; 
\draw (5.5,2.5) node {$15$}; 

\draw (6.5,6.5) node {$-9$}; 
\draw  [blue](6.5,5.5) node {$-4$}; 
\draw [blue](6.5,4.5) node {$1$}; 
\draw [red] (6.5,3.5) node {$6$}; 
\draw (6.5,2.5) node {$11$}; 

\draw (7.5,5.5) node {$-8$}; 
\draw  (7.5,4.5) node {$-3$}; 
\draw [blue](7.5,3.5) node {$2$}; 
\draw [red](7.5,2.5) node {$7$}; 
\draw (7.5,1.5) node {$12$};

\draw (8.5,4.5) node {$-7$}; 
\draw  (8.5,3.5) node {$-2$}; 
\draw [red](8.5,2.5) node {$3$}; 
\draw (8.5,1.5) node {$8$};

\draw  (9.5,3.5) node {$-6$}; 
\draw [blue](9.5,2.5) node {$-1$}; 
\draw [red](9.5,1.5) node {$4$};

\draw  (10.5,3.5) node {$-10$}; 
\draw (10.5,2.5) node {$-5$}; 
\draw [red](10.5,1.5) node {$0$};

\draw [very thick, blue] (0,10)--(1,10) --(1, 8)--(2, 8) --(2, 7) --(4,7) --(4, 6) --(6, 6) --(6, 4) --(7, 4) --(7, 3) --(9, 3) --(9, 2) --(11, 2) --(11, 1);
\end{tikzpicture}

\caption{Periodic lattice path corresponding to the $(4,5)$-invariant set $\Delta=\{0,3,4,5,\ldots\}.$ The $5$-generators of $\Delta$ are $\{0,3,4,6,7\}$ in \textcolor{blue}{blue}, and the $4$-cogenerators are $\{-4,-1,1, 2\}$ in \textcolor{red}{red}. }
\label{Figure:periodic_path}
\end{center}
\end{figure}

We follow the notation from \cite{gorsky2022generic} for the equivalence classes defined in \cite{gorsky_mazin_vazirani_2017} for $(m, n)$-invariant sets. Let $\gcd(m, n) = d \geq 1$ and $ \Theta = \{\Delta_0, \Delta_1, \dots , \Delta_{d-1} \}$ be $0$-balanced $(m/d, n/d)$-invariant sets. For every $(x_0, x_1, \dots x_{d-1}) \in \nR^{d}$ such that $\sum_{i=0}^{d-1}x_i = 0$ let 

\[
\Delta(x_0, \dots x_{d-1}) = \bigcup_{i=0}^{d-1} ( d \Delta_i + x_i).
\]
   For each $\Delta_i$ let $S_i$ be its skeleton. Consider the space of all possible shifts $x_0, \dots x_{d-1}$ and consider $\Sigma_{\Theta} \subset \{\x \in \nR^{d-1} : \sum_{i = 0}^{d-1} x_i = 0 \}$ consisting of all shifts $x_0, \dots, x_{d-1}$ for which there exists $i$ and $j$ such that
\[
(d S_i + x_i) \cap (d S_j + x_j) \neq \emptyset.
\]
See that $\Sigma_\Theta$ is a hyperplane arrangement.
If a shift $(x_0, x_1, \dots, x_{d-1})$ is integer and there is an $x_i$ in each remainder $\mod d$, then $\Delta(x_0, x_1, \dots, x_{d-1})$ is an $(m, n)$-invariant set. 

\begin{definition}
    Two $(m, n)$-invariant sets are equivalent if they correspond to the same connected component of $\Sigma_{\Theta}$. See that we get hyperplane arrangements for each tuple $\Theta$ of $(m/d, n/d)-$invariant sets.     
\end{definition}

To get all the equivalence classes of bounded and co-bounded $(m,n)$-invariant sets, one should consider all possible $d$-tuples of $0$-balanced $(m/d, n/d)$-invariant sets and for each such $d$-tuple consider the set of connected components of the complement to $\Sigma_{\Theta}$ in the space of shifts. These should be considered up to symmetry; if two of the sets are equal $\Delta_i=\Delta_j$, then switching the corresponding shift coordinates $x_i$ and $x_j$ interchanges the connected components corresponding to the same equivalence class of $(m, n)$-invariant sets. Furthermore, if $\Theta'$ is a permutation of $\Theta$ then the regions correspond to the same permutation of shifts.

In \cite{gorsky_mazin_vazirani_2017} the authors define a bijection $\cG$ from equivalence classes of $(m, n)$-invariant sets $ M_{m,n / \sim}$ to $(m, n)$-Dyck paths $Y_{m,n}$, by recording the order in which $n$-cogenerators and $m$-generators are in the skeleton $S$.

\begin{definition}

Given an $(m, n)$-invariant set $\Delta$ have skeleton $S$. Let $\cG: M_{m,n / \sim} \to Y_{m,n}$ be the map where an equivalence class of $(m, n)$-invariant sets is mapped to an $(m, n)$-Dyck path in the following way. The path of $\cG(\Delta)$ can be obtained by rearranging the skeleton of $\Delta$ in increasing order, replacing $n$-generators by up steps and $m$-cogenerators by left steps, from bottom right to top left. 

\end{definition}

\begin{proposition}[\cite{https://doi.org/10.48550/arxiv.1406.1196}, proposition 3.2]
    $\cG(\Delta)$ is an $(m, n) -$Dyck path. 
\end{proposition}

\begin{theorem}[\cite{gorsky_mazin_vazirani_2017},Theorem 1.1] \label{theorem: G bijection}
The map $\cG$ is a bijection.
\end{theorem}

\begin{example} \label{example: map G}
    Consider the two $(6, 4)$-invariant sets  \begin{align*}
        \Delta &= \nZ_{\geq 0} \setminus \{ 1, 2, 3, 5,6 , 7\} \\
        \Delta' &= \nZ_{\geq 0} \setminus \{ 1, 2, 3, 5 , 7,9,13\}.
    \end{align*}
    Whose respective skeletons are 
    \begin{align*}
        S = \{ \textcolor{red}{-4},\textcolor{blue}{0}, \textcolor{red}{2}, \textcolor{blue}{4}, \textcolor{red}{5},\textcolor{blue}{8}, \textcolor{blue}{9}, \textcolor{red}{11}, \textcolor{blue}{13}, \textcolor{blue}{17}  \} \\
        S' = \{ \textcolor{red}{-4}, \textcolor{blue}{0}, \textcolor{red}{2},\textcolor{blue}{4}, \textcolor{red}{7},\textcolor{blue}{8}, \textcolor{blue}{11}, \textcolor{red}{13}, \textcolor{blue}{15}, \textcolor{blue}{19}  \},
    \end{align*}
    where we have colored $4$-cogenerators in \textcolor{red}{red} and $6$-generators in \textcolor{blue}{blue}. These are in the same region of $\Sigma_\theta$ where $\Theta = \{\nZ_{\geq 0} \setminus \{1\}, \nZ_{\geq 0} \setminus \{1\} \}$, which means they are in the same equivalence class. This can be seen as the odd generators and cogenerators of $\Delta$ can be shifted by 2 without crossing any even generators or cogenerators, making $\Delta'$.
    Figure \ref{Figure: example of G} has the result of $\cG(\Delta)$ and $\cG(\Delta ')$.

\begin{figure}[ht]
    \begin{center}

\begin{tikzpicture}[scale=0.75]
\draw  [very thin, gray](0,0) grid (4,6);
\draw  [very thin, gray](4,0) -- (0,6);

\draw [very thick, blue] (0, 6)--(0, 4); 
\draw [very thick, red]  (0, 4)--(1, 4); 
\draw [very thick, blue] (1, 4) --(1, 2); 
\draw [very thick, red]  (1, 2)--(2, 2);
\draw [very thick, blue] (2, 2) --(2, 1); 
\draw [very thick, red]  (2,1)--(3, 1);
\draw [very thick, blue] (3, 1) --(3, 0);
\draw [very thick, red]  (3,0)--(4, 0);

\end{tikzpicture}
\caption{This is $\cG(\Delta) = \cG(\Delta ')$ from Example \ref{example: map G}. }
\label{Figure: example of G}
\end{center}
\end{figure}

\end{example}

\section{Acting on m-invariant sets}

\begin{definition}
    A subset $\Delta \subset \nZ$ is an $m$-invariant set if $\Delta + m \subset \Delta$. The set of all $m$-invariant sets is denoted by $M_m$.   
\end{definition}

\begin{definition}
We define an action of words in $[m] = \{ 0, 1, \dots m-1\}$ on the set $M_m$ of $m$ invariant sets. Let  $\{ a_0< a_1 < \dots < a_{m-1} \} = \Delta \setminus (\Delta + m)$ be the $m$-generators of $\Delta$. For $j \in [m]$ set $j \cdot \Delta = \Delta \setminus\{ a_j \}$, note that $\Delta \setminus \{ a_j \}$ is an $m$-invariant set. Given a word $w = w_{n-1}\dots w_{0}$ in $[m]$ we act on $\Delta$ one letter at a time from right to left 

\[
\w \cdot \Delta = w_{n-1}(\dots(w_1(w_0 \Delta))\dots ).
\]

\end{definition}

\begin{definition} \label{def: phi coprime}
    We define a map $\phi: M_{m}\to V^m $ from $m$-invariant sets into the Weyl Chamber by 
    setting 
    $$\phi(\Delta) = (a_0, a_1, \dots, a_{m-1}) - \frac{\sum_{j = 0}^{m-1} a_j}{m} \mathbb{1},$$
    
    where $a_0, a_1, \dots a_{m-1}$ are the $m$-generators of $\Delta$ and  $\mathbb{1} = (1, 1, \dots, 1)$.
\end{definition}

\begin{lemma}
    The map $\phi : M_{m} \to V^m$ commutes with actions of words in $[m]$ on $V^m$ and $M_{m}$. The image $\phi(\Delta)$ is the centriod of an alcove. 
\end{lemma}

\begin{proof}
  See that $\sum_{i=1}^{m-1} (a_i - \frac{\sum_{j-1}^{m-1} a_j}{m}) = 0$, thus $\phi(\Delta) \in V^m$ and $\phi$ is well defined. The $m$-generators of $i \cdot \Delta$ are $$\{a_0, \dots , a_{i-1}, a_{i+i}, \dots, a_{m}, a_i + m \},$$ so 
  $$\phi(i \cdot \Delta) = \sort((a_0, a_1, \dots,a _i + m ,\dots,  a_{m-1}) - \mathbb{1}(\sum_{j = 0}^{m-1} (a_j - 1)/m))$$ 
  sorted. This agrees with $i \cdot \phi(\Delta)$. As the set of $m$-generators of $\Delta$ contains exactly one number from each equivalence class modulo $m$, its clear that $\phi(\Delta)$ is a centriod by Definition \ref{def: Centriod}.
\end{proof}

\begin{theorem}
    If $\w \cdot \Delta = \Delta + n$ then $\w$ is an $(m, n)$-rational parking function. 
\end{theorem}

\begin{proof}
    The map $\phi$ sends shifts of $\Delta$ to the same point so $\w \cdot \phi(\Delta) = \phi( \w \cdot \Delta ) = \phi( \Delta + n ) = \phi(\Delta)$. Thus $\w$ has a fixed point in $V^m$ and is an $(m, n)$-parking function.
\end{proof}

From now on, we denote $\p$ as an $(m, n) $-parking function. 

\begin{lemma}\label{shiftset}
If $\p \cdot \Delta = \Delta + n$ then $\Delta$ is an $(m, n)$-invariant set. 
\end{lemma}

\begin{proof}

We know $\p \cdot \Delta \subset \Delta$, so $\Delta + n \subset \Delta$ thus $\Delta$ is $(m, n)$-invariant.

\end{proof}

A slightly stronger statement is true. At every step we must have an $(m, n)$-invariant set. Let $\Delta^{(i)} = p_{i-1} \dots p_0 \Delta$ where $\Delta^{(0)} = \Delta$.

\begin{corollary} \label{Corollary: Each Step mn}
If $\p \cdot \Delta = \Delta + n$ then $\Delta^{(i)}$ is an $(m, n)$-invariant set for all $0 \leq i \leq n$. 
\end{corollary}

\begin{proof}

    Consider $\Delta^{(i)}$ and  set $\p' = p_{i-1}\dots p_0 p_n \dots p_i$ then $$\p'\cdot \Delta^{(i)} =  p_{i-1}\dots p_0 p_n \dots p_ip_{i-1} \dots p_0 \Delta = p_{i-1}\dots p_0 \cdot (\Delta + n) = \Delta^{(i)} + n$$
    By \cref{shiftset} $\Delta^{(i)}$ is an $(m, n)$-invariant set. 
\end{proof}

\begin{example}

 See \cref{Figure:Acting on Set} for the action of $p = 103$ on the $(5, 3)$ invariant set $\Delta = \nZ \setminus \{ 1, 2, 4\}$. We show that $p \cdot \Delta = \Delta + 3$, at each step, we record the $5$-generators and $\Delta^{(i)}$. You can also notice that at each step $\Delta^{(i)}$ is an $(m, n)$-invariant set and we remove elements that are also $n$-generators. Both of these will happen every time you act by a word such that $\p \cdot \Delta = \Delta + n$.
\end{example}

\begin{figure}[ht]
\[
\begin{array}{ccc}
    \w & 5 -gen & \Delta^{(i)} \\
     3& \{ 0, 3, 6, \textcolor{red}{7}, 9 \} & \{ 0, 3, 5, 6, \textcolor{red}{7}, 8, 9, 10, 11, 12, 13, 14, \dots \} \\
    0 & \{ \textcolor{red}{0}, 3, 6, 9, 12 \} & \{ \textcolor{red}{0}, 3, 5, 6, 8, 9, 10, 11, 12, 13, 14, \dots \} \\
    1 & \{  3, \textcolor{red}{5}, 6, 9, 12 \} & \{  3, \textcolor{red}{5}, 6, 8, 9, 10, 11, 12, 13, 14, \dots \} \\
     & \{  3, 6, 9, 10, 12 \} & \{  3, 6, 8, 9, 10 ,11, 12, 13, 14, \dots \} \\
\end{array}
\]
\caption{ Acting on the $(5, 3)$ invariant set $\Delta = \{ 0, 3, 5, 6, 7, 8, 9, 11, 12,\dots \} $ by the parking function $\p = 103$. The $5$-generators being replaced in each step is highlighted in red. The $3$-generators of $\Delta$ are $\{ 0, 5, 7\}$ which is exactly the set of elements that are removed by $\p$.  }
\label{Figure:Acting on Set}
\end{figure}

\begin{lemma}
Let $\p = p_{n-1}, \dots p_{0}$, if $\p \cdot \Delta = \Delta + n$ and $a_{p_i}$ is the $p_i$th $m$-generator of $ \Delta^{(i)}$ then $a_{p_i}$ is an $n$-generator of $ \Delta^{(i)}$. That is, at each step, we remove elements that are also $n$-generators of $\Delta^{(i)}$. 
\end{lemma}

\begin{proof}
We know $\Delta^{(i)}$ and $\Delta^{(i + 1)}$ are $(m, n)$-invariant sets from \cref{Corollary: Each Step mn} and $\Delta^{(i)}\setminus \Delta^{(i+1)} = \{ a_{\omega_i}\}$. If $a_{p_i}$ is not an $n$-generator of $\Delta^{(i)}$ then $a_{p_i} - n \in \Delta^{(i)}$ and $a_{p_i} - n \in \Delta^{(i+1)}$ but $a_{p_i } \notin \Delta^{(i+1)}$ which contradicts $\Delta^{(i+1)}$ being an $(m, n)$-invariant set. 
\end{proof}

\begin{lemma} \label{lemma: Also n gen of Delta}
If $\p \cdot \Delta = \Delta + n$ and $a_{p_i}$ is the $p_i$th $m$-generator of $ \Delta^{(i)}$ then $a_{p_i}$ is an $n$-generator of $ \Delta$. That is at each step we remove elements that are also $n$-generators of $\Delta$ the original set. 
\end{lemma}

\begin{proof} 
 The action of a letter removes one element of $\Delta$. The set of $n$-generators are given by $\Delta \setminus (\Delta + n)$, which means none of the $n$-generators of $\Delta$ are in $\p \cdot \Delta = \Delta + n$. As $\p$ has length $n$ at each step, we must remove one of the $n$-generators of $\Delta$. 
\end{proof}

\begin{lemma} \label{ lemma: skel of pDelta }
    If $\p \cdot \Delta = \Delta + n$ then the set $\cup_{i =0}^n \mgen( \Delta^{(i)} )$ is the skeleton of $\Delta + n$ as an $(n, m)$-invariant set, note the switch in $n$ and $m$. It is the union of $m$-generators of $\Delta+ n$ and $n$-cogenerators of $\Delta + n$.
\end{lemma}

\begin{proof}
     First we show the skeleton $S$ of $\Delta + n$ is a subset of $\cup_{i =0}^n \mgen( \Delta^{(i)} )$. In Lemma \ref{lemma: Also n gen of Delta} it's shown that of $\p$ only removes the $n$-generators of $\Delta$. If $x \in \mgen(\Delta^{(i)})$ then either it gets removed at some step of the action or is in $\mgen(\Delta + n)$. If $x$ gets removed, then $x$ must have been an $n$-generator of $\Delta$. Observe that $n$-generators of $\Delta$ are $n$-cogenerators of $\Delta +n$. 
     
    \vspace{10 pt}
     
     Now we show the opposite inclusion. As  $\p \cdot \Delta = \Delta + n$ we know $\p$ must remove all $n$-generators of $\Delta$. For any $n$-generator $x$ of $\Delta$ there is some step where it is removed. That is there is some $i \in \{ 1,  \dots, n-1\}$ such that $\Delta^{(i-i)} \setminus \Delta^{(i)} = \{ x\}$ and $x \in \mgen(\Delta^{(i-1)})$

\end{proof}

    The skeleton of $\Delta$ as an $(m, n)$-invariant set $S$, and  as an $(n, m)$-invariant set $S'$ are related. See that $S' = S + n - m$.

\section{The associated monotone parking function}

In this section, we define the associated monotone parking function to an $(m, n)$-invariant set.

\begin{definition}
    A parking function $\p = p_{n-1}\dots p_0$ is monotone if $p_0 \leq p_1 \leq \dots \leq p_{n-1}$.
\end{definition}

\begin{theorem}
    Given an $(m, n)$-invariant set $\Delta$ there is a unique monotone parking function $\p$ such that $\p \cdot \Delta = \Delta + n$. We call this the associated monotone parking function to $\Delta$. 
\end{theorem}

The idea is that we create $\p$ in an iterative process by removing the $n$-generators of $\Delta$ from smallest to largest. In each step, we record the position of the removed generator within the set of current $m$-generators.

\begin{proof}

    We show that there exists a monotone parking function that removes the $n$ generators of $\Delta$ in order. Let $\{a_0< a_1 < \dots< a_{n-1} \}$ be the $n$-generators of $\Delta$. Set $p_0 = 0$, as the minimal element must be an $m$ generator and an $n$ generator. For each $i \in \{ 1 , \dots , n-1\}$ 
    we claim that $a_i$ must be an $m$ generator of $\Delta ^{(i)} = \Delta \setminus \{a_0, \dots a_{i-1} \}$. This claims implies that there exists some $0 \leq p_i < m$ such that $p_i \cdot \Delta^{(i)} = \Delta^{(i+1)}$.
    See, that $\Delta^{(i)}$ has $n$ generators 
    $$\{a_0 +n, \dots , a_{i-1} + n, a_{i}, \dots ,  a_{n-1}  \},$$
    not necessarily written in order.
    For the sake of contradiction, assume that $a_i$ is not an $m$-generator of $\Delta^{(i)}$ that is $a_i - m \in \Delta^{(i)}$. Then there must be some $n$ generator $\tilde{a}$ of $\Delta^{(i)}$ with $\tilde{a} + k n = a_i - m $ with $k\geq 0$. For some $j < i$, we must have $\tilde{a} = a_j + n$ where $a_j$ is an $n$ generator of $\Delta$. We know that $a_i - n \notin \Delta$, thus $a_i - m - n \notin \Delta$, we arrive at $a_i - m - n = a_j + kn$ the left side is not in $\Delta$, but the right side is, a contradiction. Thus $a_i$ must be an $m$ generator of $\Delta^{(i)}$. 

    \vspace{10 pt}

    Let $\mgen(\Delta^{(i)}) = \{ b_0^{(i)}, \dots b_{m-1}^{(i)} \}$, there is some $j$ such that $b_j^{(i)} = a_i$, set $p_i = j$. Then we have $p_{i-1}\dots p_0\Delta = \Delta^{(i)}$ and $p_{n-1}\dots p_0\Delta = \Delta + n$, so $\p = p_{n-1} \dots p_0$ must be an $(m, n)$-parking function as it shifts $\Delta$ by $n$. 

    \vspace{10 pt}

    We show that any monotone parking function $\p$ such that $\p\cdot \Delta  = \Delta + n$ must remove the $n$-generators of $\Delta$ from smallest to largest. From Lemma \ref{lemma: Also n gen of Delta}, we know $\p$ must remove all the $n$-generators of $\Delta$. Say $p$ did not remove the $n$-generators in order and let $a_i$ be the first $n$-generator removed out of order. Consider the $m$-generators of $\Delta^{(i)}$ which are $\{ b^{(i)}_0, b^{(i)}_1, \dots, b^{(i)}_{m-1}  \}$, as shown in the first part of this proof, there must be some $k$ with $a_i = b^{(i)}_k$. Then $p_i > k$, as $a_i$ is skipped. Then $a_i$ will remain an $m$-generator until it is removed. Moreover, $a_i$ will remain the $k$th smallest $m$-generator of $\Delta^{t}$ until it is removed, as it is the smallest $n$-generator of $\Delta$. Say it is removed on step $t$, then $p_t = k$ and $p_i > p_t$ with $i < t$. Thus, the monotone parking function $\p$ is unique.

\end{proof}

\begin{example}
    Let $\Delta = \{ 0, 3, 5, 6, 7,8, \dots \} \in M_{3, 5}$, The $5$-generators of $\Delta$ are $\{ 0, 3, 6, 7, 9\}$. At each step we remove the smallest $n$-generator of $\Delta$ that appears as an $m$-generator of $p_{i-1}\dots p_{0}\Delta$. 
    \begin{center}
    \begin{tabular}{c|ccc}
       Stage  &  $p_i$ & $p_{i-1}\dots p_{0}\Delta$ & $m$-generators    \\
        0 &0 & \{ \textcolor{red}{0}, 3, 5, 6, 7,8, \dots \}& \{\textcolor{red}{0}, 5, 7 \} \\
        1 &0 & \{ \textcolor{red}{3}, 5, 6, 7,8,9 \dots \}& \{\textcolor{red}{3}, 5, 7 \} \\
        2 &1 & \{ 5, \textcolor{red}{6}, 7,8,9,10 \dots \}& \{ 5, \textcolor{red}{6},7 \} \\
        3 &1 & \{ 5, \textcolor{red}{7},8,9, 10, 11 \dots \}& \{  5,\textcolor{red}{7},9 \} \\
        4 &1 & \{ 5,8,\textcolor{red}{9}, 10, 11, 12 \dots \}& \{  5,\textcolor{red}{9},10 \} \\
    \end{tabular}
    \end{center}

The element removed at each step is in red. See that $11100 \cdot \Delta = \{ 5, 8, 10, 11, 12, \dots\} = \Delta + 5$, so $\p = 11100$ is the unique monotone $(3, 5)$-associated monotone parking function to $\Delta = \{ 0, 3, 5, 6, 7,8, \dots \} $. 
\end{example}

There is a connection between the associated monotone parking function and the map $\cG$ from \cite{gorsky_mazin_vazirani_2017}.

\begin{lemma} \label{ lemma: G to mono }

The monotone parking function given by the row lengths of the Dyck path $\cG(\Delta)$ is the associated monotone parking function for $\Delta$. 

\end{lemma}

\begin{proof}

    Let $\{a_0, \dots a_{n-1} \}$ be the $n$-generators of $\Delta$ and 
    associated monotone parking function $ \p =p_{n-1}\dots p_{0}$. For each $i\in \{0, \dots n-1\}$ we have that $a_i $ is the $p_i$th $m$-generator of $\Delta^{(i)}$. Let $S$ be the skeleton of $\Delta + n$. 
    For each $i$ let $\mgen(\Delta^{(i)}) = \{ b_0^{(i)}, \dots , b_{m-1}^{(i)} \}$ be the $m$ generators of $\Delta^{(i)}$.

 \vspace{10 pt}
 
    For any monotone parking function we have $p_0 = 0$ as the minimal element of $\Delta$ is an $m$ and $n$ generator. So $a_0 = b_0 = b_{p_0}$ is an $m$-cogenerator of $\Delta+n$ and is the minimal element of $S$. Now consider $p_1 \geq p_0$ and $\{ b_0^{(1)}, \dots , b_{m-1}^{(1)} \}$ the $m$-generators of $p_0\Delta$. We know that $\{b_0^{(1)} , \dots b_{p_1 -1}^{(1)} \}$, a possibly empty set, will not be removed by $p_1$ or by any later letter in $\p$ as $\p$ is monotone. Thus $\{ b_0^{(1)}, \dots b_{p_1 -1}^{(1)} \}$ will be $m$-generators of $\Delta + n$. We also have $b_{p_1}^{(1)} = a_1$ which will be an $n$ cogenerator of $\Delta + n$. So $S$ starts with $\{a_0, b_0^{(1)} , \dots ,b_{p_1 -1}^{(1)}, a_1\}$.
 
  \vspace{10 pt}
 
  Consider $p_2 \geq p_1$ and $\{b_0^{(2)}, \dots , b_{m-1}^{(2)} \}$ the $m$-generators of $p_1p_0\Delta$. As before $\{ b_0^{(2)}, \dots , b_{p_2 -1}^{(2)} \}$ will be $m$-generators of $\Delta + n$. The $m$ generators $\{ b_0^{(2)}, \dots, b_{p_1 -1}^{(2)} \} = \{b_0^{(1)}, \dots , b_{p_1 -1}^{(1)}  \}$ are already accounted for. Again $b_{p_2}^{(2)} =a_2$ is a $n$-cogenerators of $\Delta + n$. Thus $S$ starts
$\{a_0, b_0^{(1)} , \dots ,b_{p_1 -1}^{(1)}, a_1, b_{p_1}^{(2)}, \dots b_{p_2-1}^{(2)} ,a_2\}$
 
 \vspace{10 pt}

We can continue this process. At the last step we will need to add $\{ b_{p_{n-1}}^{(n)}, \dots , b_{m-1}^{(n)} \}$ as they are $m$ generators of $\Delta + n$ that were not added yet. This process will determine all of $S$, we have
\[
S= \{a_0, b_0^{(1)} , \dots ,b_{p_1 -1}^{(1)}, a_1, b_{p_1}^{(2)}, \dots , b_{p_2-1}^{(2)} ,a_2, \dots ,a_i , b_{p_i}^{(i+1)} , \dots , b_{p_{i+1}-1}^{(i+1)} , a_{i+1},  \dots , a_{n-1} ,b_{p_{n-1}}^{(n)}, \dots , b_{m-1}^{(n)} \}.
\]

See that the number of $m$-generators between $a_{i-1}$ and $a_{i}$ is $p_{i} - p_{i-1}$ which is exactly $g_{i}$. The set $S$ is the skeleton of $\Delta + n$.

\end{proof}

\begin{example}
Consider the $(4, 6)$ invariant set 
\[\Delta = \{  0, 4, 5, 6, 8,9,10,11,12, \dots\}.\] For this set the $6$-generators are $\{ 0,4, 5, 8, 9, 13 \}$, The associated monotone parking function is $p = 311000$. At each step, we remove the smallest $6$-generator of $\Delta$ that appears as an $4$-generator of $\Delta^{(i)}$. 
\[
\begin{array}{ccccc}
    i & \Delta^{(i)}&\text{4-generators of } (\Delta^{(i)}) & a_i & p_i  \\
    0 & \{  0, 4, 5, 6, 8,9,10,11, 12, \dots\} & \{ \textcolor{blue}{0}, 5, 6, 11 \} & 0 & 0 \\
    1 & \{   4, 5, 6, 8,9,10,11,12, 13, \dots\} & \{ \textcolor{blue}{4}, 5, 6, 11 \} & 4 & 0 \\
    2 & \{    5, 6, 8,9,10,11,12,13,14, \dots\} & \{  \textcolor{blue}{5}, 6,8, 11 \} & 5 & 0 \\
    3 & \{     6, 8,9,10,11,12,13,14,15, \dots\} & \{  \textcolor{red}{6},\textcolor{blue}{8}, 9, 11 \} & 8 & 1 \\
    4 & \{     6,9,10,11,12,13,14,15,16, \dots\} & \{  6,\textcolor{blue}{9}, 11,12 \} & 9 & 1 \\
    5 & \{     6,10,11,12,13,14,15,16,17, \dots\} & \{  6, \textcolor{red}{11},\textcolor{red}{12},\textcolor{blue}{13} \} & 13 & 3 \\
    6 & \{     6,10,11,12,14,15,16,17,18, \dots\} & \{  6, 11,12,\textcolor{red}{17} \} & &  \\
\end{array}
\]

We have highlighted when an $6$-cogenerator of $\Delta + 6$ is removed from $\Delta$ in \textcolor{blue}{blue} and the first time an $4$-generator of $\Delta + 6$ is skipped by the action \textcolor{red}{red}. The $4$-cogenerators of $\Delta$ are $\{ -4, 1, 2, 7 \}$ and the skeleton of $\Delta$ is $\{ \textcolor{red}{-4}, \textcolor{blue}{0}, \textcolor{red}{1}, \textcolor{red}{2}, \textcolor{blue}{4}, \textcolor{blue}{5}, \textcolor{red}{7},\textcolor{blue}{8}, \textcolor{blue}{9}, \textcolor{blue}{13}  \}$.

\begin{center}

\begin{tikzpicture}[scale=0.75]
\draw  [very thin, gray](0,0) grid (4,6);
\draw  [very thin, gray](4,0) -- (0,6);

\draw [very thick, blue] (0, 6)--(0, 3); 
\draw [very thick, red]  (0, 3)--(1, 3); 
\draw [very thick, blue] (1, 3) --(1, 1); 
\draw [very thick, red]  (1, 1)--(3, 1);
\draw [very thick, blue] (3, 1) --(3, 0); 
\draw [very thick, red]  (3,0)--(4, 0);

\draw  (-.5,.5) node {$3$}; 
\draw  (-.5,1.5) node {$1$}; 
\draw  (-.5,2.5) node {$1$}; 
\draw  (-.5,3.5) node {$0$}; 
\draw  (-.5,4.5) node {$0$}; 
\draw  (-.5,5.5) node {$0$};

\end{tikzpicture}
\end{center}

\end{example}

Combining \cref{ lemma: G to mono } and \cref{theorem: G bijection} one obtains.

\begin{lemma} \label{Mono Bijection}
If $\Delta$ and $\Delta'$ have the same associated monotone parking function then $\Delta \sim \Delta'$
\end{lemma}

\begin{proof}
The $n$ cogenerators and $m$ generators in the skeleton of $\Delta$ and $\Delta'$ are in the same order, thus $\cG(\Delta) = \cG(\Delta')$ and $\Delta \sim \Delta'$ as $\cG$ is a bijection by Theorem 1.1 in \cite{gorsky_mazin_vazirani_2017}. 
\end{proof}

\section{Periodic Points} \label{sec: Periodic Points}

In this section, we study the points in $V^m$ that are not fixed by an $(m, n)$-parking functions $\p$.
The action of a parking function $\p$ on $\x$ resorts the coordinates of $\x$ several times. In this section, we discuss this sorting. 

\begin{definition}
Let $\x$ be in the interior of an alcove. We define the permutation $\sigma_{i, \x}$ of a letter $i$ acting on $\x$ and the permutation $\sigma_{\p, \x}$ of a parking function $\p$ acting on $\x$ as follows. Let $\sigma_{i, \x} \in S_m$ be such that

\[
i \cdot \x = \sigma_{i, \x} (x_0-1,\dots , x_{i-1}-1 , x_{i} + m-1, x_{i+1}-1, \dots x_{m-1}-1)
\]
That is, $\sigma_{i, \x}$ is how the coordinates are resorted when $i$ acts on $\x$.
Note that since $\x$ is in the interior of an alcove, all the coordinates of $\x$ are distinct, and furthermore, $i \cdot \x$ is also in the interior of an alcove.
The permutation of a parking function $\p = p_{m-1} \dots p_0$ on $\x$ is $$\sigma_{\p,\x} = \sigma_{p_{n-1}, p_{n-2}\dots p_0 \x}\dots \sigma_{p_1, p_0\x} \sigma_{p_0, \x}.$$ 
\end{definition}

This works with concatenation. If $\p$ and $\textbf{\textit{q}}$ are words then $\sigma_{\p, \textbf{\textit{q}}  \x}\sigma_{\textbf{\textit{q}}, \x} = \sigma_{\textbf{\textit{pq}}, \x}$. In particular if $\x$ is fixed by $\p$ then $(\sigma_{\p, \x})^\ell = \sigma_{\p^\ell, \x}$ where $\p^\ell$ is $\p$ concatenated with itself $\ell$ times. 

\begin{definition}
    Let $\x \in V^m$ for any $i \in \{ 0, \dots, m-1\}$ let 
    \[
    i \star \x = (x_0-1,\dots , x_{i-1}-1 , x_{i} + m-1, x_{i+1}-1, \dots x_{m-1}-1).
    \]
    That $i \star \x$ is the shift during the action of $i$.
\end{definition}

This is a motivating example for why we are now starting to track $\sigma_\p$. 

\begin{example} \label{exp: Parking not closer}
    Let $\w = 2100 \in PW_4^3$, then the fixed point of $\w$ is $x_\w = (-4, 0, 4)$. We consider the centers of the 3 centers of alcoves adjacent to $A_\w$ and see where $\w$ maps them. We have labeled the centers $a, b, c$ where $a = (-4, 1, 3)$, $b= (-3, -1, 4), $ and $c = (-5, 0, 5)$. We will find that $\w \cdot a = b, \w \cdot b = x_\w,$ and $\w \cdot c = x_\w$.

\begin{figure}[ht]
\begin{center}

\begin{tikzpicture}[scale=1]

\foreach \x in {0,1, 2, 3}
    \draw (\x/2, \x*\3) -- (4-\x/2,\x*\3);

\foreach \x in {1, 2, 3,4}
    \draw (\x, 0) -- (\x/2,\x*\3);

\foreach \x in {0,1, 2, 3,4}
    \draw (\x, 0) -- (4/2+\x/2,4*\3-\x*\3);

\draw[line width=.5mm] (0,0) -- (4, 0);
\draw[line width=.5mm] (0,0) -- (4/2, 4*\3);

\node[font = \small] (w) at (2, 1.35*\3) {$x_w$};

\node[font = \small] (a) at (1.5, 1.6*\3) {$a$};

\node[font = \small] (b) at (2, .6*\3) {$b$};

\node[font = \small] (c) at (2.5, 1.6*\3) {$c$};

\node[font = \small] at (0, -.2) {$\vec{0}$};

\draw [-{>[sep=-1pt]}] (a) to [bend right=45] (b);

\draw [-{>[sep=-1pt]}] (b) to  (w);

\draw [-{>[sep=-1pt]}] (w) to  (c);

\end{tikzpicture}

\caption{In this example $\w = 2100$, we consider the four points, the fixed point $x_\w = (-5, 0, 5)$ along with $a = (-4, 1, 3), b= (-3, -1, 4), $ and $c = (-5, 0, 5)$. Arrows show where each point is mapped to by the parking function $\w = 2100$. } 
\label{Figure:Large Overview}

\end{center}
\end{figure}
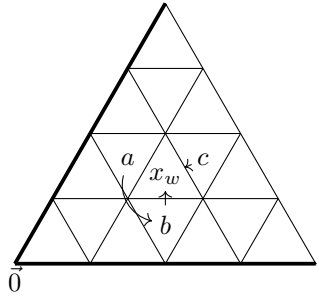

To highlight the geometry of what is happening, we will show each step of the action of $w$ on the four alcoves that contain $x_\w, a, b$ and $c$ in Figure \ref{Figure:Large Steps}. Recall that the action shifts and then resorts points. Resorting can be viewed as folding back into the Weyl Chamber. At each step we showed where we started in blue, the shift in red and the resort is the blue highlighted area of the next step.

\end{example}

\begin{figure}[ht]
\begin{center}

\begin{tikzpicture}[scale=1]

\fill
  [fill opacity=.5, fill=blue!30,draw=blue,thick] (1,2*\3) -- (3,2*\3) -- (2,0) -- cycle;

\fill
  [fill opacity=.5, fill=red!30,draw=red,thick] (1/2,\3) -- (2.5,\3) -- (1.5,-\3) -- cycle;

    \foreach \x in {1, 2,3}
    \draw (\x/2, \x*\3) -- (4-\x/2,\x*\3);

\foreach \x in {1, 2, 3}
    \draw (\x, 0) -- (\x/2,\x*\3);

\foreach \x in {1, 2, 3}
    \draw (\x, 0) -- (4/2+\x/2,4*\3-\x*\3);

\draw (3.5, \3) -- (2.5, 3*\3);

\draw[line width=.5mm] (0,0) -- (3, 0);
\draw[line width=.5mm] (0,0) -- (3/2, 3*\3);

\draw[dashed] (.5, -\3) --( 2.5 ,-\3);
\draw[dashed] (-.5, \3) --( .5 ,3*\3);

\foreach \x in {1, 2, 3}
    \draw[dashed] (-.5 + \x, -\3) -- (\x, 0);
\foreach \x in {1, 2, 3}
    \draw[dashed] (-.5 + \x, -\3) -- (\x-1, 0);
\foreach \x in { 2, 3,4}
    \draw[dashed] (-1.5 + \x/2, -\3 + \x * \3) -- (-.5 + \x/2, -\3 + \x * \3);
\foreach \x in {0,1, 2}
    \draw[dashed] (-.5 + \x/2, \3 + \x * \3) -- ( \x/2,   \x * \3);

\node[font = \small] at (0, -.2) {$\vec{0}$};

\node[circle,fill=black,scale=0.3] (w) at (2, 1.35*\3) {};

\node[font = \small] (a) at (1.5, 1.65*\3) {$a$};

\node[font = \small] (b) at (2, .65*\3) {$b$};

\node[font = \small] (c) at (2.5, 1.65*\3) {$c$};

\node[circle,fill=black,scale=0.3] (1) at (1.5, .35*\3) {};


\draw [-{>[sep=-1pt]}] (w) to node[left] {0}  (1);

\end{tikzpicture}\hspace{.3 cm}\begin{tikzpicture}[scale=1]

\fill
  [fill opacity=.5, fill=blue!30,draw=blue,thick] (1/2,\3) -- (2.5,\3) -- (2,0) -- (1, 0) -- cycle;

  \fill
  [fill opacity=.5, fill=red!30,draw=red,thick] (0,0) -- (2,0) -- (1.5,-\3) -- (.5, -\3) -- cycle;

    \foreach \x in {1, 2,3}
    \draw (\x/2, \x*\3) -- (4-\x/2,\x*\3);

\foreach \x in {1, 2, 3}
    \draw (\x, 0) -- (\x/2,\x*\3);

\foreach \x in {1, 2, 3}
    \draw (\x, 0) -- (4/2+\x/2,4*\3-\x*\3);

\draw (3.5, \3) -- (2.5, 3*\3);

\draw[line width=.5mm] (0,0) -- (3, 0);
\draw[line width=.5mm] (0,0) -- (3/2, 3*\3);

\draw[dashed] (.5, -\3) --( 2.5 ,-\3);
\draw[dashed] (-.5, \3) --( .5 ,3*\3);

\foreach \x in {1, 2, 3}
    \draw[dashed] (-.5 + \x, -\3) -- (\x, 0);
\foreach \x in {1, 2, 3}
    \draw[dashed] (-.5 + \x, -\3) -- (\x-1, 0);
\foreach \x in { 2, 3,4}
    \draw[dashed] (-1.5 + \x/2, -\3 + \x * \3) -- (-.5 + \x/2, -\3 + \x * \3);
\foreach \x in {0,1, 2}
    \draw[dashed] (-.5 + \x/2, \3 + \x * \3) -- ( \x/2,   \x * \3);

\node[font = \small] at (0, -.2) {$\vec{0}$};

\node[circle,fill=black,scale=0.3] (1) at (1.5, .35*\3) {};

\node[circle,fill=black,scale=0.3] (2) at (1, -.65*\3) {};

\draw [-{>[sep=-1pt]}] (1) to node[left] {0}   (2);


\end{tikzpicture} \hspace{.3cm}
\begin{tikzpicture}[scale=1]

  \fill
  [fill opacity=.5, fill=blue!30,draw=blue,thick] (0,0) -- (2,0) -- (1.5,\3) -- (.5, \3) -- cycle;
  
  \fill
  [fill opacity=.5, fill=red!30,draw=red,thick] (0,2*\3) -- (1,2*\3) -- (1.5,\3)  -- (-.5, \3) -- cycle;

    \foreach \x in {1, 2,3}
    \draw (\x/2, \x*\3) -- (4-\x/2,\x*\3);

\foreach \x in {1, 2, 3}
    \draw (\x, 0) -- (\x/2,\x*\3);

\foreach \x in {1, 2, 3}
    \draw (\x, 0) -- (4/2+\x/2,4*\3-\x*\3);

\draw (3.5, \3) -- (2.5, 3*\3);

\draw[line width=.5mm] (0,0) -- (3, 0);
\draw[line width=.5mm] (0,0) -- (3/2, 3*\3);

\draw[dashed] (.5, -\3) --( 2.5 ,-\3);
\draw[dashed] (-.5, \3) --( .5 ,3*\3);

\foreach \x in {1, 2, 3}
    \draw[dashed] (-.5 + \x, -\3) -- (\x, 0);
\foreach \x in {1, 2, 3}
    \draw[dashed] (-.5 + \x, -\3) -- (\x-1, 0);
\foreach \x in { 2, 3,4}
    \draw[dashed] (-1.5 + \x/2, -\3 + \x * \3) -- (-.5 + \x/2, -\3 + \x * \3);
\foreach \x in {0,1, 2}
    \draw[dashed] (-.5 + \x/2, \3 + \x * \3) -- ( \x/2,   \x * \3);

\node[font = \small] at (0, -.2) {$\vec{0}$};

\node[circle,fill=black,scale=0.3] (2) at (1, .65*\3) {};

\node[circle,fill=black,scale=0.3] (3) at (.5, 1.65*\3) {};

\draw [-{>[sep=-1pt]}] (2) to node[right] {1}  (3);


\end{tikzpicture}\hspace{.3 cm}\begin{tikzpicture}[scale=1]

  \fill
  [fill opacity=.5, fill=blue!30,draw=blue,thick] (1,2*\3) -- (1.5,\3) -- (1,0)  -- (.5, \3) -- cycle;

    \fill
  [fill opacity=.5, fill=red!30,draw=red,thick] (2,2*\3) -- (2.5,\3) -- (2,0)  -- (1.5, \3) -- cycle;

    \foreach \x in {1, 2,3}
    \draw (\x/2, \x*\3) -- (4-\x/2,\x*\3);

\foreach \x in {1, 2, 3}
    \draw (\x, 0) -- (\x/2,\x*\3);

\foreach \x in {1, 2, 3}
    \draw (\x, 0) -- (4/2+\x/2,4*\3-\x*\3);

\draw (3.5, \3) -- (2.5, 3*\3);

\draw[line width=.5mm] (0,0) -- (3, 0);
\draw[line width=.5mm] (0,0) -- (3/2, 3*\3);

\draw[dashed] (.5, -\3) --( 2.5 ,-\3);
\draw[dashed] (-.5, \3) --( .5 ,3*\3);

\foreach \x in {1, 2, 3}
    \draw[dashed] (-.5 + \x, -\3) -- (\x, 0);
\foreach \x in {1, 2, 3}
    \draw[dashed] (-.5 + \x, -\3) -- (\x-1, 0);
\foreach \x in { 2, 3,4}
    \draw[dashed] (-1.5 + \x/2, -\3 + \x * \3) -- (-.5 + \x/2, -\3 + \x * \3);
\foreach \x in {0,1, 2}
    \draw[dashed] (-.5 + \x/2, \3 + \x * \3) -- ( \x/2,   \x * \3);

\node[font = \small] at (0, -.2) {$\vec{0}$};

\node[circle,fill=black,scale=0.3] (3) at (1., 1.35*\3) {};

\node[circle,fill=black,scale=0.3] (w) at (2, 1.35*\3) {};

\draw [-{>[sep=-1pt]}] (3) to node[above] {2}  (w);


\node[font = \small] (b) at (2, .65*\3) {$b$};

\end{tikzpicture}

\caption{We track the four alcoves that contain $x_\w, a, b$ and $c$ over the action of $\w$. } 
\label{Figure:Large Steps}

\end{center}
\end{figure}
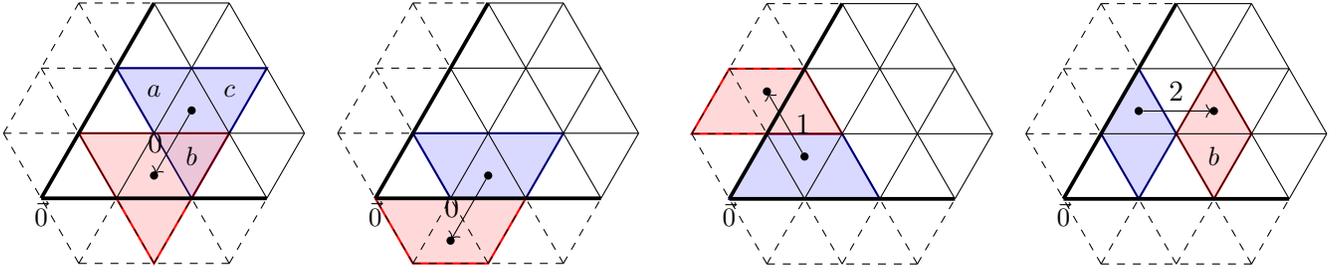

During the first step of the action $0 \star x_\w$ and $0 \star c$ are separated by $H^0_{0,1}$ after the shift and after sorting that $0 \cdot x_\w = 0 \cdot c$. See that,image $\sigma_{0, \x_\w} = id$ and $\sigma_{0, b} = s_0$, the impage of these points are getting closer to each other. Simlarly during the third step $1 \star 0 0 \cdot \x_\w$ and $1 \star 00 \cdot c $ are separated by $H_{1, 2}^0$. After the sorting, we have $100 \cdot x_\w = 100 \cdot c$. At this we have $\sigma_{100, \x_\w} = (12)(01)$ and $\sigma_{100, c} = (01)$. 

See that after applying all of $\w$ we have $\w \cdot a = b$. This also means $\| \x_\w- a \| = \| w \cdot \x  - w \cdot a \| $. To illustrate Lemma \ref{SameNorm}, see that $\sigma_{\w , \x_\w} = \sigma_{\w, a}$. It can also be seen that $\w^2 \cdot a = \x_\w$. In this section, we will show that all points outside $A_\w$ will eventually be mapped into $A_\w$. This example shows that there are some dynamics about the action that need to be studied.

\begin{lemma}\label{SameNorm}
If $\bu,\bv \in V^m$ are in the interior of alcoves $A_\bu$ and $A_\bv$ respectively, for $i \in \{0, 1, \dots m-1\}$ we have $\| \bu - \bv \| = \| i \cdot \bu - i \cdot \bv \|$ if and only if $\sigma_{i, \bu} = \sigma_{i, \bv}$. Furthermore, if $\| \bu - \bv \| = \| i \cdot \bu - i \cdot \bv \|$ then there is a linear map $T: \nR^m \to \nR^m$ such that $i|_{A_\bu} = T|_{A_\bu}$ and $i|_{A_\bv} = T|_{A_\bv}$. 
\end{lemma}

\begin{proof}

Its clear that if $i\cdot \bu$ and $i \cdot \bv$ are resorted in the same way then $\| \bu - \bv \| = \| i \cdot \bu - i \cdot \bv \|$. Now suppose that  $\sigma_{i, \bu} \neq \sigma_{i, \bv}$. Let $\tilde{\bv} = \sigma_{i,\bu}(v_1 -1, \dots, v_i + m -1, \dots, v_m -1)$ be $\bv$ sorted in the same way as $i \cdot \bu$. Then $\tilde{\bv}$ is not in order, so there is some $j$ such that $\tilde{v}_j >\tilde{v}_{j+1}$. We can consider the difference between $\| i \cdot \bu - \tilde{\bv} \|$ and $\| i \cdot \bu - s_j\tilde{\bv}\|$ where $s_j\tilde{\bv}$ has $v_j$ swapped with $v_{j+1}$. 
\begin{center}

\tikzset{every picture/.style={line width=0.75pt}} 

\begin{tikzpicture}[x=0.5pt,y=0.5pt,yscale=-1,xscale=1]

\draw  [dash pattern={on 0.84pt off 2.51pt}]  (196.83,208.42) -- (504.33,11.92) ;
\draw    (328.33,59.92) -- (331.33,253.92) ;
\draw    (328.33,59.92) -- (358.33,104.92) ;
\draw    (358.33,104.92) -- (388.33,149.92) ;
\draw    (388.33,149.92) -- (331.33,253.92) ;
\draw  [dash pattern={on 0.84pt off 2.51pt}]  (329.83,124.92) -- (388.33,149.92) ;

\draw (161.03,210.27) node [anchor=north west][inner sep=0.75pt]    {$x_{j} =x_{j+1}$};
\draw (313.83,39.07) node [anchor=north west][inner sep=0.75pt]    {$\tilde{\bv}$};
\draw (304.63,251.77) node [anchor=north west][inner sep=0.75pt]    {$i\cdotp \bu$};
\draw (394.33,134.07) node [anchor=north west][inner sep=0.75pt]    {$s_{j} \ \tilde{\bv}$};
\draw (312,111.4) node [anchor=north west][inner sep=0.75pt]    {$L$};

\end{tikzpicture}
\end{center}
Let $L$ be the point on the intersection of the hyperplane $x_j = x_{j+1}$ and the line segment from $i \cdot \bu$ to $\tilde{\bv}$. Then

\begin{align*}
    \| i \cdot \bu - \tilde{\bv} \| &=  \| i \cdot \bu - L \| +  \| L- \tilde{\bv} \| \\
    &= \| i \cdot \bu - L \| +  \| L- s_j \tilde{\bv} \|.
\end{align*}
By the triangle inequality, which is strict as $\tilde{\bv}$ and $ s_j \tilde{\bv}$ do not coincide 
\begin{align*}
    \| i \cdot \bu - s_j \tilde{\bv} \| &< \| i \cdot \bu- L \| +  \| L- s_j \tilde{\bv} \| \\
    &= \| i \cdot \bu - \tilde{\bv} \|. 
\end{align*}
As one resorts, the coordinates of $\tilde{\bv}$ one goes through several of these reflections. Thus $\| \bu - \bv \| > \| i \cdot \bu - i \cdot \bv \|$. 

Now, assume that $\| \bu - \bv \| = \| i \cdot \bu - i \cdot \bv \|$ we know that the action is piecewise linear on alcoves. Let $T$ and $T'$ be the linear map such that $i|_{A_\bu} = T|_{A_\bu}$ and $i|_{A_\bv} = T'|_{A_\bv}$. From the definition of the action $T(\x) = \sigma_{\bu, i}(\x - \y )$ and $T'(\x) = \sigma_{\bv,i}(\x - \y)$ where $\y = (1, \dots , 1, 1-m, 1, \dots 1 )$. As $\sigma_{i, \bu} = \sigma_{i, \bv}$ we have $T = T'$.

\end{proof}

This continues to words as well. If $\| \w \cdot \x - \w \cdot \y \| = \|\x - \y \|$ then $\sigma_{\w, \x} = \sigma_{\w, \y} $, moreover, for each $0 \leq i \leq n$ we have $\sigma_{ w_i, x^{(i)} } = \sigma_{w_i, y^{(i)}}$. Furthermore, there is a linear map $T: V^m \to T^m$ such that $\w|_{A_x} = T|_{A_\x}$ and $\w|_{A_\y} = T|_{A_y}$ where $A_\x$ and $A_\y$ are alcoves such that $\x \in A_\x$ and $\y \in A_\y$.

\begin{lemma} \label{lemma: same perm by letter}
    If $\x$ and $\y$ are from the interior of the same alcove, then for all $i \in \{ 0, \dots, m-1\}$ $\sigma_{i, \x} = \sigma_{i, \y}$. 
\end{lemma}

\begin{proof}
    Recall the action of a letter restricted to an alcove is a linear isometry to the alcove it maps to. Thus all points in the alcove must be sorted the same way, thus $\sigma_{i, \x} = \sigma_{i, \y}$
\end{proof}

This extends to words. If $\x$ and $\y$ are in the same alcove, then $\sigma_{\w, \x} = \sigma_{\w, \y}$ for all words $\w \in [m]^n$. The following lemma is similar to Lemma 2.5 from \cite{mccammond_thomas_williams_2019}.

\begin{lemma} \label{lemma: cords in part}
   If $\x$ is a fixed point of $\p$ and in the interior of some alcove, then the coordinates of $\x $ can be partitioned into sets $P_0, \dots P_{d-1}$ of size $m/d$ such that for all $0 \leq i \leq d-1 $ and $a \neq b \in P_i$

\begin{align*}
    &1)\quad a - b \in d \nZ & &2) \quad a - b \notin m \nZ
\end{align*}
Furthermore, if follows that for $i \neq j$, $a \in P_i$, and $b \in P_j$ that $a - b \notin d \nZ$.
\end{lemma}

\begin{proof}

    As $\x$ is in the interior of an alcove for any coordinates $x_i, x_j$ we have $x_i - x_j \notin m \nZ$.  As $\x$ is fixed by $\p$ for each coordinate $x_j$ there is some $i$ and $k$ such that $x_j - n + mk = x_i$. This means $x_j - x_i \in d \nZ$ while $x_j - x_i \notin m \nZ$. Given a coordinate $x_j$, consider that $m/d$ cosets 
    \[
    x_j + m \nZ, \quad x_j + d + m \nZ ,\quad x_j + 2d + m \nZ,\quad \dots ,\quad x_j+ (m-d) + m \nZ.
    \]
    Each set contains exactly one coordinate of $\x$. Then each coordinate belongs to a set of $m/d$ coordinates formed in this fashion. This creates the partition $P_0, P_1, \dots, P_{d-1}$, which have both properties.

    Furthermore, if $i \neq j$ with $a \in P_i$ and $b \in P_j$ we cannot have $a - b \in d \nZ$, as each coset contains exactly one coordinate of $\x$.

\end{proof}

Consider $\x = (x_1, x_2, \dots, x_{m-1})$ is a fixed point of $\p$ and in the interior of some alcove. Let $P_0, \dots, P_{d-1}$ be a partition of coordinates of $\x$ from Lemma \ref{lemma: cords in part}. For every $j$ and $i$ such that $x_j \in P_i$ we have $x_{\sigma_{\p, \x}(j)}$ is the unique element of $P_i$ such that
    \[
    x_{\sigma_{\p, \x}(j)} - x_j + n \in m\nZ
    \]
\begin{corollary} \label{Cor: fixed cycle}  
    If $\x$ is fixed by $\p$ then the permutation $\sigma_{\p, \x}$ is $d$ disjoint cycles of size $m/d$.
\end{corollary}

\begin{proof}
    From the discussion above, if $x_j \in P_i$ then $x_{\sigma_{\p, \x}(j)} \in P_i$, Applying the same logic several times we arrive at 
    \[
    P_i = \{ x_j, x_{\sigma_{\p, \x}(j)}, x_{\sigma_{\p, \x}\sigma_{\p, \x}(j)} = x_{\sigma_{\p^2, \x}(j)}, \dots x_{\sigma_{\p^{m/d -1}, \x }(j)}\}.
    \]

    Thus $\sigma_{\p, \x}$ permutes the coordinates in $P_i$ in a long cycle. There are $d$ parts $P_i$, thus $\sigma_{\p, \x}$ is $d$ disjoint cycles of size $m/d$.
\end{proof}

\begin{lemma}
    If $\| \p \bu - \p \bv \| = \| \bu - \bv\|$ then for every $\epsilon > 0$ there exists $\bu'$ and $\bv'$ such that $\bu' \in B_\epsilon(\bu)$, $\bv' \in B_\epsilon(\bv)$ such that $\bu'$ and $\bv'$ are in the interior of alcoves and $\sigma_{\p, \bu'} = \sigma_{\p, \bv'}$. Furthermore, there are alcoves $A_\bu$ and $A_\bv$ such that $\bu \in A_\bu$ and $\bv \in A_\bv$ and a linear map $T$ such that $\p|_{A_\bu} = T|_{A_\bu}$ and $\p|_{A_\bv} = T|_{A_\bv}$.
\end{lemma}

\begin{proof}
    If $\bu$ and $\bv$ are already in the interior of alcoves, then by Lemma 5.2 
    we can take $\bu' = \bu$ and $\bv' = \bv$ and have $\sigma_{\p, \bu'} = \sigma_{\p, \bv'}$. If $\bu$ or $\bv$ is contained in some hyperplane $H \in \cH$, we consider the line segment $L$ connecting $\bu$ and $\bv$, as $\| \p \bu - \p \bv \| = \| \bu - \bv\|$ we know, $L$ is mapped isometrically by $\p$ to its image, which must also be a line segment. 
    For any $\epsilon > 0$ one can pick $\bu'' \in B_{\epsilon/2}(\bu) \cap L$ and $\bv'' \in B_{\epsilon/2}(\bv) \cap L$ such that if $\bu''$ or $\bv''$ is contained in hyperplane $H_{i, j}^k$ then $L \subset H_{i, j}^k$ and each alcove containing $\bu''$ also contains $\bu$, similarly, each alcove containing $\bv''$ also contains $\bv$. 
    Let $\w \in B_{\epsilon/2}(0)$ be any point such that $\bu' = \bu'' + \w$ and $\bv' =\bv'' + \w$ are in the interiors of alcoves $A_\bu$ and $A_\bv$ such that $\bu, \bu', \bu'' \in A_\bu$ and $\bv, \bv', \bv'' \in A_\bv$.

    Let $L'$ be the line segment from $\bu'$ to $\bv'$.We claim that $\p$ maps $L'$ 
    to its image isometrically. Indeed, suppose that $\| \p \bu' - \p \bv' \| < \| \bu' - \bv' \|$. Let $k$ be such that $\|  p_k \dots p_0 \bu' - p_k \dots p_0 \bv' \|  = \|\bu' - \bv'  \|$ and $\|  p_{k+1} p_k \dots p_0 \bu' - p_{k+1}p_k \dots p_0 \bv' \|  < \|\bu' - \bv'  \|$. During the action of $p_{k+1}$ before sorting, the images of $\bu'$ and $\bv'$ must be on opposite sides of a hyperplane $H_{i, i+1}^0 =\{x_i = x_{i+1}\}$ for some $i$. Thus the alcoves containing the images of $\bu'$ and $\bv'$ are separated by $H_{i, i+1}^0$. The images of $\bu''$ and $\bv''$ cannot lie on $H_{i, i+1}^0$, as they are in the interior of alcoves. Thus the images of $\bu''$ and $\bv''$ are separated by $H_{i, i+1}^0$. Then $p_{k+1}$ necessarily sort the images of $\bu''$ and $\bv''$ in different ways, thus $\|\p \cdot \bu'' - \p \cdot \bv''\| < \| \bu'' - \bv'' \|$. But $\bu'', \bv'' \in L$ so this contradicts that $L$ is mapped isometrically.

    Thus $L'$ is mapped isometrically, which means that $\bu' \in B_\epsilon(\bu)$, $\bv' \in B_\epsilon(\bv)$ such that $\bu'$ and $\bv'$ are in the interior of alcoves, and $\sigma_{\p, \bu'} = \sigma_{\p, \bv'}$ by Lemma 5.2.

    Furthermore, by lemma 5.2 there is a linear map $T$ such that $\p|_{A_\bu} = T|_{A_\bu}$ and $\p|_{A_\bv} = T|_{A_\bv}$ as $\bu'$ and $\bv'$ are in the interior of alcoves with  $\| \p \bu' - \p \bv' \| = \| \bu' - \bv' \|$.
\end{proof}

From now on assume that $m$ and $n$ are coprime, $\p$ is an $(m, n)$-parking function, the unique fixed point of $\p$ is $\x_p$, and $A_p$ is the alcove containing $\x_p.$ Note that according to Lemma 2.9 from \cite{gorsky_mazin_vazirani_2014} $x_p$ must be the centroid of $A_p$.

\begin{lemma}\label{Periodic}
Given an $(m, n)$-parking function $\p$, for all $\x \in V^m$ there is some $k \in \nZ_{\geq 0}$ such that $\p^{k+m}\x = \p^k\x$. That is $\{\p^k\x\}_k$ will eventually become fixed or periodic, with period dividing $m$.  
\end{lemma}

\begin{proof}
Consider an $\x \in V^m$. Acting by a letter or word is a weak contraction, thus 
\[
\|\p \cdot \x - \x_p\| \leq \| \x - \x_p \|
\]
for all $\x \in V^m$. First consider $\x$ in the interior of some alcove. If we consider the set $\{\p^k \cdot \x \}_k$, because the action is a weak contraction, this set can only intersect finitely many alcoves in $V^m$ as the distance to $\x_p$ is bounded. Thus, for some $k < j$ point $\x^k, \x^j$ are in the same alcove of $V^m$, say $\p^k\x, \p^j\x \in A$. We know $\p$ isometrically maps alcoves to alcoves, thus $\p^{(j-k)}: A \to A$. It is well known that alcoves have dihedral symmetry $D_{2m}$, so $\p^{(j-k)}: A \to A$ can only be one of $2m$ possible maps. Thus $\p^{(j-k)2m + k}\x = \p^k \x$ and repeated action of $\p$ must become periodic.

When the action becomes periodic, the distance to $\x_p$ will stop decreasing and $\| \p^k \x - \x_p\| = \| \p^{k+j}\x - \x_p \|$ for all $j \geq 0$. By \cref{SameNorm} $\sigma_{\p, \x_p} = \sigma_{\p, \p^{k+j}\x}$ and by Corollary \ref{Cor: fixed cycle} $\sigma_{\p, \x_\p}$ a cycle of length $m$. Thus $\sigma_{\p, \p^{k+j}\x} = \sigma_{\p, \p^{k+j}\x_\p}$ is the same cycle of length $m$ for all $j$. Then 
$$\sigma_{\p, x_p}^m = \sigma_{\p^m, x_p} = id = \sigma_{\p, \p^{k+m}\x_\p}$$

The linear map $\p^m|{A_{\x_p}}$ is the identity map, thus $\p^m|{\p^jA_{\x}}$ must also be the identity map by Lemma 5.6. Which implies that $\p^{j + m} \cdot \x = \p^j \cdot \x$.

Now consider $\x$ on the boundary of some alcove $A'$. Let $\x'$ be a point in the interior of $A'$. Then for any word $\w$ we have $\|\w \x - \w \x' \| = \| \x - \x' \|$, as each letter maps alcoves to alcoves isometrically. By the above, there is some $j$ such that $\p^m|_{\p^jA'}: p^j A' \to p^j A'$ is the identity map and $\p|_{\p^j A'}$ has period $m$.

\end{proof}

Now we can classify all the periodic points of a given parking function. 

\begin{lemma}\label{lemma: On Hi i+1}
If $\x_p$ is fixed by a parking function $\p=p_{n-1}\dots p_0 $, then there is some $0 < k< n$ and $0 \leq i < m$ where $p_{k-1}\dots p_0 \x_p = \x_p^{(k)}$ such that $|x^{(k)}_{p,i} - x^{(k)}_{p,i+1}| = 1$. Which means if $A_p$ is the alcove containing $\x_p$ then $p^k\dots p_0A_p$ has a facet contained in the hyperplane $H_{i, i+1}^0 = \{x_i = x_{i+1} \}$ for some $i$.

\end{lemma}

\begin{proof}
By Corollary \ref{Cor: fixed cycle}, we know that $\sigma_{\p, \x_p}$ is a cycle of length $m$, and all coordinates of $\x_p$ have distinct remainders $\mod m$. Let $x_i = 0 \mod m$ and $x_j = 1 \mod m$. Without loss of generality, say that $x_i < x_j$ and thus $i < j$. As $\sigma_{p, \x_p}$ is a cycle, there is a minimum $\ell$ such that $(\sigma_{\p, \x_p})^\ell(i) > (\sigma_{\p, \x_p})^\ell(j) $. Then $(\sigma_{\p, \x_p})^{\ell-1}(i) < (\sigma_{\p, \x_p})^{\ell-1}(j)$ and there is a minimal $0 \leq k \leq n-1$ such that 

\[
\sigma_{p_{k-1} \dots p_0 p^{\ell -1}, \x}(i) > \sigma_{p_{k-1} \dots p_0 p^{\ell -1}, \x}(j).
\]

Set $\hat{i} = \sigma_{p_{k-2} \dots p_0 p^{\ell}} (i)$, $\hat{j} = \sigma_{p_{k-2} \dots p_0 p^\ell} (j) $, and $\hat{x} = p_{k-2} \dots p_0 p^\ell \cdot \x$. See that $\hat{i} < \hat{j}$ and $\hat{x}_{\hat{i}} <\hat{x}_{\hat{j}}$ and $\sigma_{p_{k-1}, \hat{x}}(\hat{i}) > \sigma_{p_{k-1}, \hat{x}}(\hat{j})$. 
As $\hat{\x}_{\hat{i}} - \hat{\x}_{\hat{j}}$ changes sign from negative to positive, it must be that $p_{k-1} = \hat{i}$. Acting by a letter does not change the remainders $\mod m$, thus $\hat{x}_{\hat{i}} + 1 \equiv \hat{x}_{\hat{j}} \mod m$. Acting adds $m$ to a coordinate, thus $\hat{x}_{\hat{i}} + m > \hat{x}_{\hat{j}}$, that means that $\hat{x}_{\hat{j}} = \hat{x}_{\hat{i}} + 1$. Since all the coordinates are integer, distinct, and ordered, we must have $\hat{j} = \hat{i} + 1$. This means that the alcove that contains $\hat{x}$ has a facet contained in $H_{i, i+1}^0$. 

\end{proof}

\begin{lemma} \label{lemma: x to fixed alcove}
Given $\p$ a parking function, let $\x_p$ be its fixed point in $V^m$, and $A_p$ be the alcove containing $x_p$. The set of periodic points of $\p$ is exactly $A_p \setminus \x_p$ all points of $A_p \setminus \x_p$ have period $m$.
\end{lemma}

\begin{proof}
If $\x$ is in the interior of $A_p$, then $\sigma_{\p, \x_p} = \sigma_{\p, \x}$ which is a cycle of length $m$. Thus for any point $\x \in A_p$ we have $p^m \cdot \x = \x$, thus all points in the closure are periodic.

For any point $\x$ not in the alcove. By \cref{Periodic} there is some $k$ such that $\|\p^k \x - \x_p \| = \|\p^{k+j}\x - \x_p \|$ for all $j \geq 0$. Assume $\p^k \x$ is not in the alcove that contains $\x_p$. Then by \cref{lemma: On Hi i+1} there is an $\ell$ such that $|x_i^{(\ell)} - x_{i+1}^{(\ell)}| = 1$, which means $p_{\ell-1}\dots p_1p_0 x_p$ is in an alcove with a facet contained in $H_{i, i+1}^0 = \{x_i = x_{i+1}\}$. Let $\x_p' = p_{\ell-1}\dots p_1p_0 \x_p$ and $\x' = p_{\ell-1}\dots p_1p_0 \x $, then $\p ' = p_{\ell-1}\dots p_0p_{n-1}\dots p_{\ell}  $ fixes $\x_p'$ and $\{\p'{}^k \x'\}_k$ becomes periodic. If we consider the points $\{ \p'^k\x' , \p'^{k+1}\x' , \p'^{k+2}\x', \dots , \p'^{k+n-1}\x' \}$ there is $t$ where $\p'^{k+t}\x'$ and $\x_p'$ are separated by $H_{i, i_1}^0$. In this sorting step $\p'^{k+t}\x'$ must have gotten closer to $\x_p'$. Thus $\{ \p^k \x\}$ must eventually be in $A_p$. 

\end{proof}

This last lemma shows that if $\x$ is not in the alcove fixed by $\p$ then $|\p^m\x - \x_p  | < |\x - \x_p|$. Because alcoves are mapped isometrically to alcoves, we get the following:

\begin{corollary} \label{corollary: x0 converge to fixed point }
    The centroid of the fundamental alcove $\x_0 \in A_0 $ will converge to the fixed point of $\x_p \in A_p$ under the action of $\p$. i.e. $\p^k \cdot \x_0 \to \x_p$.
\end{corollary}

\section{Fixed points of parking functions}

    It was shown in \cite{mccammond_thomas_williams_2019} that the set of all fixed points of an $(m, n)$-parking function is a convex set of dimension $d-1$ in $V^m$.
    Recall the notation from Section \ref{SubSec: Alcoves and the action of a Parking Function}. Let $ \Theta = (\Delta_0, \Delta_1, \dots , \Delta_{d-1} )$ be a tuple of $0$-balanced $(m/d, n/d)$-invariant sets. For each $\Theta$, let $V_\Theta$ be a copy of the hyperplane $\{\x \in \nR^{d} : \sum_{i = 0}^{d-1} x_i = 0 \}$, 
    and consider the hyperplane arrangement $\Sigma_{\Theta} \subset V_\Theta$ consisting of all shifts $x_0, \dots, x_{d-1}$ for which there exists $i$ and $j$ such that
    \[
    (d S_i + x_i) \cap (d S_j + x_j) \neq \emptyset,
    \]
    where $S_i$ is the skeleton of $\Delta_i$.

    A point $(x_0, x_1, \dots, x_{d-1})$ in a region of $\Sigma_\Theta$ corresponds to the shift of $\Theta = (\Delta_0, \Delta_1, \dots, \Delta_{d-1} )$, specifically 
\[
\Delta(x_0, \dots x_{d-1}) = \bigcup_{i=0}^{d-1} ( d \Delta_i + x_i).
\]
We consider the configuration space 
$$  W = \{ (\Theta , (x_0, \dots , x_{d-1})) : \Theta = (\Delta_0, \dots, \Delta_{d-1}) \text{ and  } (x_0, \dots , x_{d-1}) \notin \Sigma_\Theta \}.$$

One can extend the map $\phi:M_{m, n} \to V^m$  (Definition \ref{def: phi coprime} from section $3$) to $W$. Consider a point $(\Theta,(x_0, \dots , x_{d-1}) ) \in W$. For each invariant set $\Delta_i$ of the tuple $\Theta = (\Delta_0, \dots, \Delta_{d-1})$, let $A_i$ be the set of $m$-generators say $A_i: = \{ a_{i, 0}, a_{i, 1} , \dots, a_{i, m/d -1} \}$. Then $\phi : W \to V^m$ is given by
\[
    \phi( (\Theta,(x_0, \dots , x_{d-1}) ) ) = \sort\left(\bigcup_{i =1}^{d-1} d A_i + x_i\right) - \mathbb{1} k,
\]
where
\[
k = \frac{d}{m}\sum_{i = 0, j = 0}^{d-1,m/d-1} a_{i,j} + \frac{1}{d} \sum_{i = 0}^{d-1} x_i.
\]

The union is a set of $m$ distinct numbers thought of as a sorted point in $V^m$. The $m$-cogenerators are part of the skeleton of each $\Delta_i$, and thus do not intersect; the $m$-generators are a shift by $m/d$ from $m$-cogenerators and thus all coordinates of $\phi( (\Theta,(x_0, \dots , x_{d-1}) )  )$ are distinct.

We have a result similar to Lemma \ref{ lemma: skel of pDelta } in this setting. Let $\cord( (x_0, \dots , x_{m/d -1}) ) = \{ x_0, \dots x_{m/d -1}\}$ be the coordinates of a point as a set. Recall that $S_i$ is the skeleton of $\Delta_i$ as an $(m,n)$-invariant set, now let $S_i'$ the skeleton of $\Delta_i$ as an $(n, m)$-invariant set. Note that this switch of $m$ and $n$ also occurs in Lemma \ref{ lemma: skel of pDelta }. 

\begin{lemma} \label{lemma: Skel gen}
    Let $\phi((\Theta, \x))$ be a fixed point of $\p$ and for each $i$ let $M_i =  p_{i-1}\dots p_0\phi( (\Theta , \x )) $. Then
    \[
    S = \bigcup_{i =0}^{n}  \left( \cord( M_i ) + i  \right)
    \]
     is the skeleton 
     \[
     S' = \bigcup_{i=0}^{d-1} dS_i' + x_i + n.
     \] 
\end{lemma}

\begin{proof}
    With an unbalanced action, we have $\p \cdot \phi(\Theta, \x)  = \phi(\Theta, \x) + n$.
    If $y \in \cord(\phi(\Theta, \x) +n)\subset S$ then for some $j$ we have $y \in d A_j + x_j + n$, that is $y$ corresponds to an $m$-generator of $\Delta_j$ and $y \in S'$. 

    The unbalanced action replaces a coordinate $x_j \in \cord(\x)$ in $V^m$ by $x_j +m$ and then sorts. If $y \in \cord(S) \setminus \cord((\phi(\Theta, \x)  + n))$ then for some $i \in \{ 0, 1, \dots, n -2\}$ we have $y \in \cord(p_i \dots p_0 \phi(\Theta, \x))$. Then for some $j$ and $k$ we have $y = da + km$ where $a \in A_j $ is an $m$-generator of $\Delta_j$. Then $y + n = da + km + n \in d \Delta_j + x_j + n$, but $y \notin  d \Delta_j + x_j + n$ because $y \notin \cord((\phi(\Theta, \x)  + n))$. That is $y$ is an $n$-cogenerator of $\Delta_j + n$. Thus $y \in S'$. 
    
\end{proof}

\begin{lemma} \label{ Lemma: Same region }
    If $(\Theta,(x_0, \dots x_{d-1}) ) \in W$ and $(\Theta, (y_0, \dots y_{d-1})) \in W$ are from the same region of $\Sigma_\theta$ and $\p \in PW_m^n$ is such that $\p$ fixes $\phi((\Theta,(x_0, \dots x_{d-1})))$, then $\p$ fixes $\phi((\Theta, (y_0, \dots y_{d-1})))$. 
\end{lemma}

\begin{proof}
   
    Let $(\Theta, (x_0, \dots x_{d-1}) ) = (\Theta, \x)$ and $(\Theta, (y_0, \dots y_{d-1})) = (\Theta , \y)$ be as in the statement of the lemma. For each $i \in \{0, \dots, d-1 \}$, let $A_i = \{ a_{i, 0}, a_{i, 1} , \dots, a_{i, n/d -1} \}$ be the $m$-generators of $\Delta_i$. If $\p$ does not fix $\phi((\Theta, \y))$, then during the action the coordinates must have been sorted differently. Consider the first time they are not sorted the same way. That is some $t$ such that 
    $$\sigma_{p_{t-1}\dots p_0,\phi((\Theta, \x))} = \sigma_{p_{t-1}\dots p_0,\phi((\Theta, \y))} $$ but 
    $$\sigma_{p_tp_{t-1}\dots p_0, \phi((\Theta, \x))} \neq \sigma_{p_tp_{t-1}\dots p_0, \phi((\Theta, \y))} .$$ That means there is some $(i_1, j_1, k_1)$ and $(i_2, j_2, k_2)$ such that 
    \begin{align*}
        da_{i_1, j_1} + x_{i_1} + k_1m + m < da_{i_2, j_2} + x_{i_2} + k_2 m
        \\
        da_{i_1, j_1} + y_{i_1} + k_1m + m > da_{i_2, j_2} + y_{i_2} + k_2 m.
    \end{align*}

    Visually changing the shifts from $\x$ to $\y$ changes how elements sort.

    \begin{center}
        \includegraphics[scale = .2]{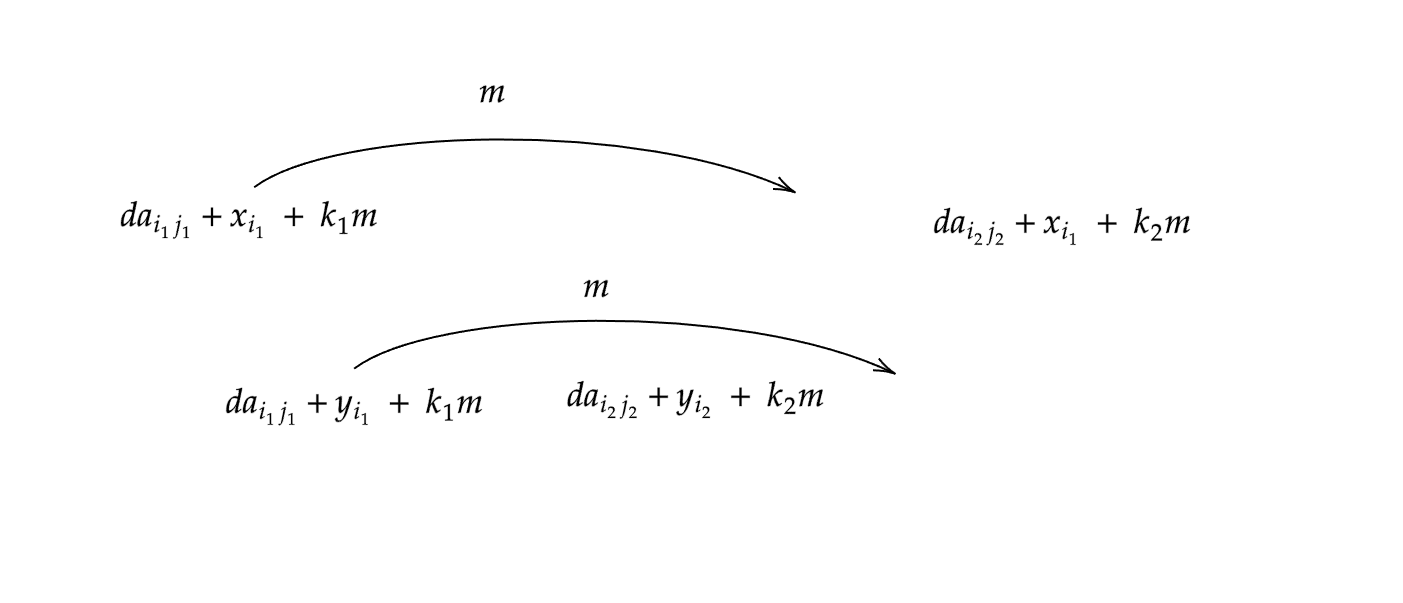}
    \end{center}

    By lemma \ref{lemma: Skel gen}, we know $ \bigcup_{i =0}^n p_{i-1}\dots p_0\phi( \Theta , \x ) = \bigcup_{i=0}^{d-1} dS_i' + x_i + n = S$ where $S_i'$ is the skeleton of $\Delta_i$ as an $(n, m)$-invariant set. We know that $da_{i_1, j_1} + x_{i_1} + k_1m + m \in S$. This inequality is saying that $\bigcup_{i=0}^{d-1} dS_i' + x_i + n$ and $\bigcup_{i=0}^{d-1} dS_i' + y_i + n$ are in a different order. But that means that $\phi(\Theta, \x)$ and $\phi(\Theta, \y)$ must be in different connected components of the complement to $\Sigma_\Theta$.

\end{proof}

\begin{lemma}
    If $(\Theta,(x_0, \dots x_{d-1}) ) ,(\Theta, (y_0, \dots y_{d-1})) \in W$ are from the same region of $\Sigma_\theta$ and $\p \in PW_m^n$ such that $\p$ fixed $\phi((\Theta,(x_0, \dots x_{d-1})))$, then for each $k \in \{ 0, 1, \dots, n-1\}$ the coordinates of $p_k\dots p_0\phi((\Theta, (y_0, \dots y_{d-1})))$ are distinct.
\end{lemma}

\begin{proof}

    By Lemma \ref{ Lemma: Same region } $\p$ also fixes $\phi(\Theta, (y_0, \dots y_{d-1}))$. By Lemma \ref{lemma: Skel gen} the skeleton $S'$ has unique coordinates. This means for each $k$ the coordinates of $p_k\dots p_0\phi((\Theta, (y_0, \dots y_{d-1})))$ are distinct.

\end{proof}

A region of $\Sigma_\Theta$ can be fixed by multiple parking functions, and some parking functions fix multiple regions. These regions have a unique monotone parking function that fixes every $\Delta(x_0, \dots, x_{d-1})$ in that region. 

The image of a region of $\Sigma_\Theta$ is the set of fixed points for the monotone parking function that region corresponds to. The image of all of the $\Sigma_\Theta$ is an arrangement of $(d-1)$ affine subspaces, which is the union of all fixed points of all parking functions.

\begin{example} Consider $m = 3$ and $n = 4$ and $\Delta = \nZ_{\geq 0}$. All of the parking functions $\p$ such that $\p \cdot \Delta = \Delta + n$ are  

$$\{0000, 0001, 0011, 0010, 0100, 0002, 0012\}.$$

Now let $m = 6, n =4$ and $\Theta = ( \Delta_0, \Delta_1  )$ where $\Delta_0 = \{ -1, 1, 2, 3, 4, 5, \dots \}$ and $\Delta_1 = \nZ_{\geq 0}$. Then $S_0 = \{-3, -1, 0, 1, 3\}$ and $S_2 = \{-2, -1, 0, 1,2 \}$  Consider 
\begin{align*}
    &\Delta(-3.5, 3. 5) = \{ -5.5, -1.5, 0.5, 2.5, 3.5, 4.5, 6.5, \dots\} \text{  and} \\
    &\Delta(-4.5. 4.5) = \{ -6.5, -2.5, -0.5,   1.5, 3.5, 4, 5, 5.5, 6.5, 7.5, \dots  \}.
\end{align*}
See that $(\Theta, (-3,5, 3.5))$ and $(\Theta, (-4,5, 4.5))$ are in different connected components of the complement to $\Sigma_\Theta$ becasue 
\[
(2 S_0 -3) \cap (2S_1 + 3) = \{ -9, -5, -3, -1, 3\} \cap \{ -1, 1, 3,5, 7 \} = \{-1,3 \} \neq \emptyset.
\]
These sets correspond to the points
\begin{align*}
    &\x_0 = \phi(\Theta, (-3,5, 3.5)) = (-7, -3, 1, 2, 4, 6) \text{  and} \\
    &\x_1 = \phi(\Theta, (-4,5, 4.5)) = ( -8.5, -4.5, -0.5, 2.5, 4.5, 6.5 ). 
1\end{align*}

Both $\x_0$ and $\x_1$ are fixed by the $(6, 4)$-parking function $p =  2140$. 

\end{example}

\begin{corollary}
    An $(m, n)$-parking function will fix the image of a union of regions from $\Sigma_\Theta$ under $\phi$. 
\end{corollary}

\begin{proof}
    If $\p$ is a parking function it must fix some $\Delta$. If a parking function fixes any $\Delta$ in a region by \cref{ Lemma: Same region } it fixed the whole region. Thus $\p$ fixes a union of regions in $\Sigma_\theta$. 
\end{proof}

\section{Inverting the map $SP$}

In \cite{gorsky_mazin_vazirani_2014} Gorsky, Mazin, and Vazirani defined a map $SP : {}^m\widetilde{S_n} \to PF_n^m$, named after the Pak-Stanley map.

\[
SP_{u}(i) = \# \{  j > i | 0 < u(i) - u(j) < m \}
\]
 Gorsky, Mazin, and Vazirani showed that it is bijective for $m = kn \pm 1$, conjectured that it is bijective for all coprime $m$ and $n$, and provided a conjectured algorithm to invert $SP$. In \cite{mccammond_thomas_williams_2019} McCammond, Thomas, and Williams proved that $SP$ is a bijection for all coprime $m$ and $n$. 
 Here we present two algorithms to invert $SP$. The algorithm from \cite{gorsky_mazin_vazirani_2014} is presented in Algorithm \ref{algo: gorsky_mazin_vazirani} and we translate that algorithm to the notation from \cite{mccammond_thomas_williams_2019} in Algorithm \ref{algo: mccammond_thomas_williams}.
 In this section, we show both algorithms create the same intermediate function and that function becomes affine periodic. By affine periodic we mean there is some $N > 0$ such that for all $x > N$ we have $U(x + n) = U(x) + n$. This leads to an inverse to the map $SP$.

\begin{algo} \label{algo: gorsky_mazin_vazirani}
    Computing $SP^{-1}$ from \cite{gorsky_mazin_vazirani_2014}. 
\end{algo}

Given $\p \in PF_{n}^m$. First, we extend $\p$ periodically $p(i + kn) = p(i)$ for $k \in \nZ$ and $0 \leq i \leq n-1$. We construct an injective function $U: \nN \to \nN$ as follows: Think about inputting outputs one at a time on the table
\[
\begin{tabular}{c|cccc}
     i& 0 & 1 & 2 & \dots  \\
     \p(i)& \p(0) & \p(1) &\p(2) & \dots \\
     U(i) & U(0) & U(1) & U(2) & \dots 
\end{tabular}
\]
Initial step: Place $0$ under the leftmost $0$. That is we set $U ( min\{ j : p(j) = 0 \} ) = 0$.

\vspace{10 pt}

 General Step:  We want to place $\alpha \in \nN$ on the table. We assume we have already placed \\$\{ 0, 1, 2, \dots , \alpha-1 \}$. We place $\alpha$ in the leftmost position $i$ with the following 2 properties: 
\begin{enumerate}
    \item $\alpha$ is to the right of $\alpha -m$. That is $i > U^{-1}(\alpha - m)$

    \item If $U(i ) = \alpha$ then $\p(i) =\# \{ \beta | \alpha - m < \beta < \alpha, U^{-1}(\beta) > i \}$
\end{enumerate}

If we never place a value in the $i$th position we set $U(i) = \infty$.

\vspace{10 pt}

\vspace{10 pt}

The condition that $\alpha$ is placed to the right of $\alpha - m$ insures that $U^{-1}(\alpha-m) < U^{-1}(\alpha) $ which we will call being $m$-restricted, following the same property from affine permutations. See that we can always place $\alpha$, as there will be a non-minimal position where $p(i) = 0$ we can place $\alpha$.

    It is useful to track how far positions are away from having property 2 of the general step. 

\begin{definition}
    For all $i$ such that $U(i) > \alpha -1$ let 
    \[
    t_\alpha(i) = p(i) - \# \{ \beta | \alpha - m < \beta < \alpha, U^{-1}(\beta) > i \}
    \]
    See that $t_0(i) = p(i)$ and the general step can be rephrased as: Place $\alpha$ in the leftmost position such that $t_{\alpha} (i) = 0$ and $i > U^{-1}(\alpha - m)$
\end{definition}

\begin{lemma} \label{lem: talphaform}
    When $t_\alpha(i)$ and $t_{\alpha - 1}(i)$ are both defined, we have $ t_{\alpha -1}(i) - t_{\alpha }(i)  \in \{ 0, 1 \}$
\end{lemma}

\begin{proof}
    See that
    \begin{align*}
        t_{\alpha -1}(i) - t_{\alpha }(i) &= - \# \{ \beta | \alpha-1  - m < \beta < \alpha -1, U^{-1}(\beta) > i \}  + \# \{ \beta | \alpha - m < \beta < \alpha , U^{-1}(\beta) > i \} \\
        &=  - \# \{ \beta | \beta = \alpha - m, U^{-1}(\beta) > i \} + \# \{ \beta | \beta = \alpha - 1  , U^{-1}(\beta) > i \}.
    \end{align*}
    Each of these sets is either empty or is exactly one element. Thus $| t_{\alpha - 1}(i) - t_{\alpha}(i) | \leq 1$.
    We have

    \[
    t_{\alpha}(i) =   \begin{cases}
                        t_{\alpha-1}(i)+1  &  i > U^{-1}(\alpha ) \text{ and } i < U^{-1}(\alpha - m) \\
                        t_{\alpha-1}(i)-1  & i < U^{-1}(\alpha) \text{ and } i > U^{-1}(\alpha - m) \\
                        t_{\alpha-1}(i) & \text{else}.
                    \end{cases}
    \]
     If $t_{\alpha -1}(i) - t_{\alpha}(i) = -1$ then $U^{-1}(\alpha - m) > i$ and $i > U^{-1}(\alpha)$ which contradicts that we place $\alpha$ to the right of $\alpha - m$.
\end{proof}

\begin{lemma}
    While defined $t_\alpha(i)$ is non-negative.
\end{lemma}

\begin{proof}
    If $t_{\alpha -1}(i) = 0$ and we were to try to place $\alpha$ such that $t_{\alpha}(i) = -1$ then $i <U^{-1}(\alpha)$ and $i > U^{-1}(\alpha - m)$ by Lemma \ref{lem: talphaform}, but this means $i$ had both property 1 and 2 and is more left than where we placed $\alpha$. Thus $t_{\alpha}(i)$ is non-negative. 
\end{proof}

    \begin{lemma} \label{lem: U ininity}
        If $U(i) = \infty$ then there is some $\beta$ with $t_{\beta}(i) = 0$ and $p(i) = m-1$
    \end{lemma}

    \begin{proof}
    For some large $\alpha$ we have placed the last $m-1$ numbers to the right of $i$ then 
    \[
    t_{\alpha}(i) = p(i) -  \# \{ \beta | \alpha - m < \beta < \alpha, U^{-1}(\beta) > i \} \leq m-1 - (m-1) = 0. 
    \]
    As $p(i)$ is bounded by $m-1$. Thus $p(i) = m-1$. 
    \end{proof}

\begin{lemma}
    In Algorithm \ref{algo: gorsky_mazin_vazirani}, for every $i \in \nZ_{\geq 0}$ we have $U(i) < \infty$. That is there are no gaps left by this algorithm.
\end{lemma}

\begin{proof}
    For the sake of contradiction, let $i \in \nZ_\geq 0$ such that $U(i) = \infty$. Then $p(i) = m-1$, by Lemma \ref{lem: U ininity}. Let $\beta$ be the minimal such that $t_\beta( i) = 0$. When scanning to place $\beta$, we must see $\{\beta -1, \dots , \beta - m +1 \}$ must be placed to the right of $i$. If we do not place at $i$ it must be that $\beta- m$ is also to the right of $i$. But then $t_{\beta - 1}( i) = 0$, as $\{ \beta -2, \dots, \beta - m \}$ are all to the right of $i$.
\end{proof}

\vspace{10 pt}

An alternative to this inverse is the following.

\begin{algo} \label{algo: mccammond_thomas_williams}
    Computing $SP^{-1}$ in the language of \cite{mccammond_thomas_williams_2019}
\end{algo}

Given $\p \in PF_{n}^m$. First, we extend $\p$ periodically $p(i + kn) = p(i)$ for $k \in \nZ$ and $0 \leq i \leq n-1$. We will construct a function $V: \nZ_{\geq 0} \to \nZ_{\geq 0}$ inductively as follows: 

\vspace{10 pt}

0th step: let $i = 0, S_0 = \{0, 1, \dots m-1 \}$. 

\vspace{10 pt}

 For a finite set $S \subset \nZ$, let $S[k]$ denote that $k$th element of $S$ when sorted from smallest to largest.  

\vspace{10 pt}
 
$i$th step for $i \geq 0$: Let $\alpha = S_i[p(i)]$, set $V(i) = \alpha$ and $S_{i+1} = (S_i \setminus \{\alpha\}) \cup \{ \alpha + m\}$. 

\vspace{10 pt}

In summary, this algorithm is starting at the point $(0, 1,\dots, m-1)$ repeatedly acting by $p$ one letter at a time without balancing and recording the coordinate $p_i$ acted on.

\vspace{10 pt}

\begin{lemma}
    For all $i \in \nZ_{\geq 0}$ there is some $j$ with $V(j) = i$. That is there are no gaps in the output of $V$. 
\end{lemma}

\begin{proof}
        See that, as $p \in PF_m^n$ we must have some $i$ with $p(i) = 0$. So in Algorithm \ref{algo: mccammond_thomas_williams}, we set $V(i)$ equal to the minimal number of set $S_i$. This means that there cannot be some number that stays in $S_i$ forever. 
\end{proof}

\begin{example}
    Consider the parking function $\p = 10002 \in PF_5^3$. First we work out finding $SP^{-1}(p)$ using Algorithm \ref{algo: gorsky_mazin_vazirani}.

    \begin{center}

        \begin{tabular}{c|cccccccccccccccccccc}
        $i$&    0  & 1 & 2 & 3 & 4 & 5 & 6 & 7 & 8 & 9 & 10 & 11 & 12 & 12 & 14 & 15 & 16 & 17 & 18 & 19   \\
        $p(i) =$  \textcolor{red}{$t_0(i)$} &   \textcolor{red}{2}    & \textcolor{red}{0} &  \textcolor{red}{ 0} &  \textcolor{red}{ 0} &  \textcolor{red}{ 1} & \textcolor{red}{ 2} &  \textcolor{red}{ 0} &  \textcolor{red}{ 0} &  \textcolor{red}{ 0} & \textcolor{red}{ 1} & \textcolor{red}{ 2} &  \textcolor{red}{ 0} &  \textcolor{red}{ 0} &  \textcolor{red}{ 0} & \textcolor{red}{ 1} & \textcolor{red}{ 2} &  \textcolor{red}{ 0} &  \textcolor{red}{ 0} &  \textcolor{red}{ 0} & \textcolor{red}{ 1}   \\
        $U(i)$/\textcolor{red}{$t_\alpha(i)$}&  \textcolor{red}{1}    & 0 &  \textcolor{red}{ 0} &  \textcolor{red}{ 0} &  \textcolor{red}{ 1} & \textcolor{red}{ 2} &  \textcolor{red}{ 0} &  \textcolor{red}{ 0} &  \textcolor{red}{ 0} & \textcolor{red}{ 1} & \textcolor{red}{ 2} &  \textcolor{red}{ 0} &  \textcolor{red}{ 0} &  \textcolor{red}{ 0} & \textcolor{red}{ 1} & \textcolor{red}{ 2} &  \textcolor{red}{ 0} &  \textcolor{red}{ 0} &  \textcolor{red}{ 0} & \textcolor{red}{ 1}   \\
        &  \textcolor{red}{ 0}    & 0 & 1 & \textcolor{red}{ 0} &  \textcolor{red}{ 1} & \textcolor{red}{ 2} &  \textcolor{red}{ 0} &  \textcolor{red}{ 0} &  \textcolor{red}{ 0} & \textcolor{red}{ 1} & \textcolor{red}{ 2} &  \textcolor{red}{ 0} &  \textcolor{red}{ 0} &  \textcolor{red}{ 0} & \textcolor{red}{ 1} & \textcolor{red}{ 2} &  \textcolor{red}{ 0} &  \textcolor{red}{ 0} &  \textcolor{red}{ 0} & \textcolor{red}{ 1}  \\
        &   2   & 0 & 1 & \textcolor{red}{ 0} &  \textcolor{red}{ 1} & \textcolor{red}{ 2} &  \textcolor{red}{ 0} &  \textcolor{red}{ 0} &  \textcolor{red}{ 0} & \textcolor{red}{ 1} & \textcolor{red}{ 2} &  \textcolor{red}{ 0} &  \textcolor{red}{ 0} &  \textcolor{red}{ 0} & \textcolor{red}{ 1} & \textcolor{red}{ 2} &  \textcolor{red}{ 0} &  \textcolor{red}{ 0} &  \textcolor{red}{ 0} & \textcolor{red}{ 1}  \\
        &   2   & 0 & 1 & 3 &   \textcolor{red}{ 1} & \textcolor{red}{ 2} &  \textcolor{red}{ 0} &  \textcolor{red}{ 0} & \textcolor{red}{ 0} & \textcolor{red}{ 1} & \textcolor{red}{ 2} &  \textcolor{red}{ 0} &  \textcolor{red}{ 0} &  \textcolor{red}{ 0} & \textcolor{red}{ 1} & \textcolor{red}{ 2} &  \textcolor{red}{ 0} &  \textcolor{red}{ 0} &  \textcolor{red}{ 0} & \textcolor{red}{ 1}  \\
        &   2   & 0 & 1 & 3 & \textcolor{red}{ 0}  & \textcolor{red}{ 1} & 4 &   \textcolor{red}{ 0} & \textcolor{red}{ 0} & \textcolor{red}{ 1} & \textcolor{red}{ 2} &  \textcolor{red}{ 0} &  \textcolor{red}{ 0} &  \textcolor{red}{ 0} & \textcolor{red}{ 1} & \textcolor{red}{ 2} &  \textcolor{red}{ 0} &  \textcolor{red}{ 0} &  \textcolor{red}{ 0} & \textcolor{red}{ 1}  \\
        &   2   & 0 & 1 & 3 &  5& \textcolor{red}{ 1}  &4 &  \textcolor{red}{ 0} & \textcolor{red}{ 0} & \textcolor{red}{ 1} & \textcolor{red}{ 2} &  \textcolor{red}{ 0} &  \textcolor{red}{ 0} &  \textcolor{red}{ 0} & \textcolor{red}{ 1} & \textcolor{red}{ 2} &  \textcolor{red}{ 0} &  \textcolor{red}{ 0} &  \textcolor{red}{ 0} & \textcolor{red}{ 1}  \\
        &   2   & 0 & 1 & 3 &  5& \textcolor{red}{ 1} & 4 & 6  & \textcolor{red}{ 0} & \textcolor{red}{ 1} & \textcolor{red}{ 2} &  \textcolor{red}{ 0} &  \textcolor{red}{ 0} &  \textcolor{red}{ 0} & \textcolor{red}{ 1} & \textcolor{red}{ 2} &  \textcolor{red}{ 0} &  \textcolor{red}{ 0} &  \textcolor{red}{ 0} & \textcolor{red}{ 1}  \\
        &   2   & 0 & 1 & 3 &  5&  \textcolor{red}{ 0}  & 4 & 6 &  7&  \textcolor{red}{ 1} & \textcolor{red}{ 2} &  \textcolor{red}{ 0} &  \textcolor{red}{ 0} &  \textcolor{red}{ 0} & \textcolor{red}{ 1} & \textcolor{red}{ 2} &  \textcolor{red}{ 0} &  \textcolor{red}{ 0} &  \textcolor{red}{ 0} & \textcolor{red}{ 1}  \\
        &   2   & 0 & 1 & 3 &  5& 8 & 4 & 6 &  7&  \textcolor{red}{ 1} & \textcolor{red}{ 2} &  \textcolor{red}{ 0} &  \textcolor{red}{ 0} &  \textcolor{red}{ 0} & \textcolor{red}{ 1} & \textcolor{red}{ 2} &  \textcolor{red}{ 0} &  \textcolor{red}{ 0} &  \textcolor{red}{ 0} & \textcolor{red}{ 1}  \\
        &   2   & 0 & 1 & 3 &  5& 8 & 4 & 6 &  7& \textcolor{red}{ 0} & \textcolor{red}{ 1} & 9 & \textcolor{red}{ 0} &  \textcolor{red}{ 0} & \textcolor{red}{ 1} & \textcolor{red}{ 2} &  \textcolor{red}{ 0} &  \textcolor{red}{ 0} &  \textcolor{red}{ 0} & \textcolor{red}{ 1}  \\  
        &   2   & 0 & 1 & 3 &  5& 8 & 4 & 6 &  7& 10 & \textcolor{red}{ 1} & 9 & \textcolor{red}{ 0} &  \textcolor{red}{ 0} & \textcolor{red}{ 1} & \textcolor{red}{ 2} &  \textcolor{red}{ 0} &  \textcolor{red}{ 0} &  \textcolor{red}{ 0} & \textcolor{red}{ 1}  \\
        &   2   & 0 & 1 & 3 &  5& 8 & 4 & 6 &  7& 10 &\textcolor{red}{ 1}  & 9 & 11&   \textcolor{red}{ 0} & \textcolor{red}{ 1} & \textcolor{red}{ 2} &  \textcolor{red}{ 0} &  \textcolor{red}{ 0} &  \textcolor{red}{ 0} & \textcolor{red}{ 1}  \\
        &   2   & 0 & 1 & 3 &  5& 8 & 4 & 6 &  7& 10 & \textcolor{red}{ 0} & 9 & 11&  12&   \textcolor{red}{ 1} & \textcolor{red}{ 2} &  \textcolor{red}{ 0} &  \textcolor{red}{ 0} &  \textcolor{red}{ 0} & \textcolor{red}{ 1}  \\
        &   2   & 0 & 1 & 3 &  5& 8 & 4 & 6 &  7& 10 & 13 & 9 & 11&  12&  \textcolor{red}{ 1} & \textcolor{red}{ 2} &  \textcolor{red}{ 0} &  \textcolor{red}{ 0} &  \textcolor{red}{ 0} & \textcolor{red}{ 1}  \\
        &   2   & 0 & 1 & 3 &  5& 8 & 4 & 6 &  7& 10 & 13 & 9 & 11&  12& \textcolor{red}{0}  &\textcolor{red}{1}  & 14 &  \textcolor{red}{ 0} &  \textcolor{red}{ 0} & \textcolor{red}{ 1}  \\
        &   2   & 0 & 1 & 3 &  5& 8 & 4 & 6 &  7& 10 & 13 & 9 & 11&  12& 15 & \textcolor{red}{1} & 14 &  \textcolor{red}{ 0} &  \textcolor{red}{ 0} & \textcolor{red}{ 1}   \\
        &   2   & 0 & 1 & 3 &  5& 8 & 4 & 6 &  7& 10 & 13 & 9 & 11&  12& 15 & \textcolor{red}{1} & 14 & 16 & \textcolor{red}{0} & \textcolor{red}{1}  \\
        &   2   & 0 & 1 & 3 &  5& 8 & 4 & 6 &  7& 10 & 13 & 9 & 11&  12& 15 & \textcolor{red}{0} & 14 & 16 & 17  & \textcolor{red}{1}   \\
        &   2   & 0 & 1 & 3 &  5& 8 & 4 & 6 &  7& 10 & 13 & 9 & 11&  12& 15 & 18 & 14 & 16 & 17  &   \textcolor{red}{1}\\
        \end{tabular}
        \end{center}

We also go through Algorithm \ref{algo: mccammond_thomas_williams} to find $SP^{-1}(p)$

    \begin{center}

    \begin{tabular}{c|c|c|c}
        i & p(i) & $S_i$ & V(i) \\
        0 &  2 & (0, 1, 2) & 2 \\
        1 &  0 & (0, 1, 5) & 0 \\
        2 &  0 & ( 1, 3, 5) & 1 \\
        3 &  0 & ( 3, 4, 5) & 3 \\
        4 &  1 & (  4, 5, 6) & 5 \\
        5 &  2 & (  4, 6, 8) & 8 \\
        6 &  0 & (  4, 6, 11) & 4 \\
        7 &  0 & (  6,7,  11) & 6 \\
        8 &  0 & ( 7,9,  11) & 7 \\
        9 &  1 & ( 9, 10, 11) & 10 \\
        10 &  2 & ( 9, 11,13) & 13 \\
        11 &  0 & (  9,11,16) & 9 \\
        12 &  0 & (  11,12,16) & 11 \\
        13 &  0 & (  12,14,16) & 12 \\
        14 & 1 & (  14,15,16) & 15 \\
        15 & 2 & (  14,16,18) & 18 \\
        16 & 0 & (  14,16,21) & 14 \\
        17 & 0 & (  16,17,21) & 16 \\
        18 & 0 & (  17,19,21) & 17 \\
        19 & 1 & (  19,20,21) & 20 \\
    \end{tabular}
        
    \end{center}
    In both examples, see that the algorithms output the same balanced affine permutation based on the last $n$ entries. In this case, the output is $u = [3, 0, 1, 2, 5]$. Further in these examples, we can see that $U(i) = V(i)$ for all $i$. 
\end{example}

\begin{lemma} \label{Algo same }
    Algorithm \ref{algo: gorsky_mazin_vazirani} and Algorithm \ref{algo: mccammond_thomas_williams} have the same output, that is $U(i) = V(i)$ for all $i$.
\end{lemma}

\begin{proof}

     Both of these algorithms impose the restriction that if $U(i) = \alpha -m$ and $U(j) = \alpha$ then $i < j$. If we know $U(i) \mod m$, we can reconstruct $U(i)$ from this. For Algorithm \ref{algo: gorsky_mazin_vazirani}, we start with $U(0)  = p_0$ as you need to place $\{0, 1, \dots, p_0 -1 \}$ to have $p_0$ inversions at $U(0)$. In Algorithm 2, we set $V(0) = \{0, 1, \dots, m-1 \}[p_0] = p_0$. Say $U(j) =V(j) $ for all $j < i$, we consider $U(i)$ and $V(i)$. 

    \vspace{10 pt}

    On the general step of Algorythm \ref{algo: gorsky_mazin_vazirani} if we placed $\alpha$ in the $i$th position, then 
    \begin{enumerate}
        \item $\alpha$ is the smallest element in its congruence class $\mod m$ that was not placed earlier. Furthermore all smaller numbers in $\alpha$'s congruence class $\mod m$ were placed to the left of position $i$. Therefore $\alpha \in S_i$. 
        \item Exactlly $p_i$ numbers from $\{ \alpha -1. \dots, \alpha - m + 1\}$ were placed to the right of position $i$. Let $R$ be the set of these numbers. For every $r \in R$ there is a unique element $r' \in S_i$ such that $m | r - r'$ and $r' \leq r < \alpha $. And vice versa, that is for every element $s' \in S_i$ such that $s' < \alpha $ there is a unique $s \in R$ such that $m | s - s'$ and $s ' \leq s < \alpha$. 
    \end{enumerate}
    Therefore, $\alpha = S_i[p_i] = V(i)$ which means $U(i) = V(i)$.

\end{proof}

Note that for Lemma \ref{Algo same } we only use that fact that $p$ is a parking function to know that there is some $i$ with $p(i) = 0$. Both these algorithms output the same function as long as there is some $i$ where $p(i) = 0$ even if $p$ is not a parking function. But when $p$ is not a parking function, neither of these algorithms will become affine periodic.

\begin{corollary}
    If $\p$ is a parking function, both of these algorithms become affine periodic.
\end{corollary}

\begin{proof}
     We show there is some $N$ such that for $i > N$ we have $U(i + n) = U(i) + n$. The set $S_0$ can be thought of as the point $\x = (0, 1, 2\dots , m-1) \in V^m$, then the general step is acting without balancing by the letter $p(i)$. That means $S_{in} $ is a shift of $\p^i\cdot(0, 1, \dots, m-1)$. By \cref{lemma: x to fixed alcove} we know $\p^i\cdot(0, 1, \dots, m-1)$ will converge to the alcove that contains the fixed point of $\p$. Because $(0, 1, 2, \dots, m)$ is a centroid of an alcove, it will converge to the fixed point in a finite number of steps. Once $\p^i\cdot(0, 1\dots, m-1)$ converges to the fixed point, the algorithm will become affine periodic in $n$ steps or less. 
\end{proof}

Note that this resolves conjecture 7.9 from \cite{gorsky_mazin_vazirani_2014}.

\pagebreak

\printbibliography

@misc{gorsky_mazin_vazirani_2014, 
 title={Affine permutations and rational slope parking functions}, 
 url={https://arxiv.org/abs/1403.0303}, 
 journal={arXiv.org}, 
 author={Gorsky, Eugene and Mazin, Mikhail and Vazirani, Monica},
 year={2014},
 month={3}
 }

@article{gorsky_mazin_vazirani_2017,
 title={Rational dyck paths in the non relatively prime case},
 volume={24},
 DOI={10.37236/6901}, 
 number={3}, 
 journal={The Electronic Journal of Combinatorics}, 
 author={Gorsky, Eugene and Mazin, Mikhail and Vazirani, Monica},
 year={2017}
 }

@misc{mccammond_thomas_williams_2019,
 title={Fixed points of parking functions},
 url={https://arxiv.org/abs/1901.02906},
 journal={arXiv.org},
 author={McCammond, Jon and Thomas, Hugh and Williams, Nathan},
 year={2019},
 month={1}
 }

@misc{https://doi.org/10.48550/arxiv.1406.1196,
  doi = {10.48550/ARXIV.1406.1196},
  
  url = {https://arxiv.org/abs/1406.1196},
  
  author = {Armstrong, Drew and Loehr, Nicholas A. and Warrington, Gregory S.},
  
  title = {Sweep maps: A continuous family of sorting algorithms},
  
  publisher = {arXiv},
  
  year = {2014},
  
  copyright = {arXiv.org perpetual, non-exclusive license}
}

@misc{gorsky2022generic,
      title={Generic curves and non-coprime Catalans}, 
      author={Eugene Gorsky and Mikhail Mazin and Alexei Oblomkov},
      year={2022},
      eprint={2210.12569},
      archivePrefix={arXiv},
      primaryClass={math.AG}
}

@article{Konheim_G_Weiss_Occupancy_Discipline,
author = {Konheim, Alan G. and Weiss, Benjamin},
title = {An Occupancy Discipline and Applications},
journal = {SIAM Journal on Applied Mathematics},
volume = {14},
number = {6},
pages = {1266-1274},
year = {1966},
doi = {10.1137/0114101},

URL = { 
    
        https://doi.org/10.1137/0114101
    
    

},
eprint = { 
    
        https://doi.org/10.1137/0114101
    
    

}

}

@article{Mazin_2017, title={Multigraph hyperplane arrangements and parking functions}, volume={21}, DOI={10.1007/s00026-017-0368-7}, number={4}, journal={Annals of Combinatorics}, author={Mazin, Mikhail}, year={2017}, month={10}, pages={653–661}}

@article{Shi_1987, title={SIGN types corresponding to an affine Weyl group}, volume={s2-35}, DOI={10.1112/jlms/s2-35.1.56}, number={1}, journal={Journal of the London Mathematical Society}, author={Shi, Jian-Yi}, year={1987}, month={2}, pages={56–74}}

@article{Sommers_2005, title={b-stable ideals in the Nilradical of a Borel subalgebra}, volume={48}, DOI={10.4153/cmb-2005-043-4}, number={3}, journal={Canadian Mathematical Bulletin}, author={Sommers, Eric N.}, year={2005}, month={9}, pages={460–472}}

@article{Stanley_1998, title={Hyperplane arrangements, parking functions and tree inversions}, DOI={10.1007/978-1-4612-4108-9_19}, journal={Mathematical Essays in honor of Gian-Carlo Rota}, author={Stanley, Richard P.}, year={1998}, pages={359–375}}

@misc{Aval_Bergeron_2015, title={Interlaced rectangular parking functions}, url={https://arxiv.org/abs/1503.03991}, journal={arXiv.org}, author={Aval, Jean-Christophe and Bergeron, François}, year={2015}, month={3}}

@article{ARMSTRONG2015159,
title = {Sweep maps: A continuous family of sorting algorithms},
journal = {Advances in Mathematics},
volume = {284},
pages = {159-185},
year = {2015},
issn = {0001-8708},
doi = {https://doi.org/10.1016/j.aim.2015.07.012},
url = {https://www.sciencedirect.com/science/article/pii/S0001870815002509},
author = {Drew Armstrong and Nicholas A. Loehr and Gregory S. Warrington}
}

@article{Hopkins_Perkinson_2015, title={Bigraphical arrangements}, volume={368}, DOI={10.1090/tran/6341}, number={1}, journal={Transactions of the American Mathematical Society}, author={Hopkins, Sam and Perkinson, David}, year={2015}, month={4}, pages={709–725}}

@article{Fishel_Vazirani_2010, title={A bijection between dominant shi regions and core partitions}, volume={31}, DOI={10.1016/j.ejc.2010.05.014}, number={8}, journal={European Journal of Combinatorics}, author={Fishel, Susanna and Vazirani, Monica}, year={2010}, month={12}, pages={2087–2101}}

@article{gorsky_mazin_2013, title={Compactified Jacobians and q,t-Catalan numbers, I}, volume={120}, DOI={https://doi.org/10.1016/j.jcta.2012.07.002}, number={1}, journal={Journal of combinatorial theory. Series A}, publisher={Elsevier BV}, author={Gorsky, Evgeny and Mazin, Mikhail}, year={2013}, month={1}, pages={49–63} }

\end{document}